\documentclass[11pt,a4paper]{article}

\usepackage{amsfonts,amsmath,amssymb,amsthm} 
\usepackage{geometry} 
\usepackage{graphicx} 
\usepackage{subcaption}
\makeatletter
\providecommand*{\input@path}{}
\edef\input@path{{images/}\input@path}
\makeatother
\graphicspath{{.}{images/}}
\usepackage{array}
\usepackage{booktabs}
\usepackage{appendix}
\usepackage{bm} 
\usepackage{bbm}
\usepackage{mathrsfs} 
\usepackage[colorlinks=true,linkcolor=blue]{hyperref}

\newcommand{\R}{\mathbb{R}} 
\newcommand{\C}{\mathbb{C}} 
\newcommand{\Z}{\mathbb{Z}} 

\newcommand{\norme}[1]{\left\Vert #1\right\Vert}

\newtheorem{lem}{Lemma}
\newtheorem{teo}{Theorem}
\newtheorem*{teo*}{Theorem}
\newtheorem{prop}{Proposition}

\theoremstyle{definition}
\newtheorem{defi}{Definition}
\newtheorem{exmp}{Example}

\theoremstyle{remark}
\newtheorem{rem}{Remark}

\DeclareMathOperator{\sn}{sn}
\DeclareMathOperator{\dn}{dn}
\DeclareMathOperator{\cn}{cn}

\DeclareMathOperator{\ch}{ch}

\DeclareMathOperator{\Span}{span}

\DeclareMathOperator{\supp}{supp}

\title{Boundary correlations for the $Z$-invariant Ising Model}
\author{Tristan Pham-Mariotti}

\begin{document}

\maketitle

\begin{abstract}
Consider the natural graph associated to a rhombus tiling of a polygonal region
in the plane. The spin correlations between boundary vertices of this graph in
the $Z$-invariant Ising model do not depend on the choice of the rhombus tiling
but only on the region. We provide a matrix formula depending on the region
which allows practical computations of boundary correlations in this setting,
extending the results of Galashin in the critical case.
\end{abstract}

\section{Introduction}

Let $R$ be a polygonal region in the plane and let $\mathbb{T}$ be a rhombus
tiling of $R$ such as in Figure~\ref{region_1}. There is a natural graph
$G_{\mathbb{T}}=(V,E)$ assigned to it, whose edges are diagonals of rhombi, see
Figure~\ref{regions_graph}. The \emph{Ising model} on this graph is defined as
follows. A \emph{spin configuration} $\sigma$ of $G_{\mathbb{T}}$ is a function
of the vertices of $G_{\mathbb{T}}$ with values in $\{-1,1\}$. The probability
of a spin configuration $\sigma$ is given by the \emph{Ising Boltzmann measure}
\[ \mathbb{P}(\sigma):=\frac{1}{Z}\prod_{e=xy\in E}\exp(j_{xy}\sigma_x\sigma_y) \]
where
\[ Z := \sum_{\sigma\in \{-1,1\}^{V}}\prod_{e=xy \in E}\exp(j_{xy}
\sigma_u\sigma_v), \]
is the \emph{Ising partition function} and $(j_e)_{e\in E}$ are positive real
numbers, called \emph{coupling constants}. Let $n\geq 1$ and let $b_1,\dotsc,
b_n$ be the vertices of $G_{\mathbb{T}}$ on the boundary of $R$ listed in
counterclockwise order. In this article, we are interested in computing the
\emph{boundary spin correlations} between $b_i$ and $b_j$, denoted by $\langle
\sigma_{b_i}\sigma_{b_j}\rangle$, for all $i,j$ in $[n]$ where $[n]$ is the
notation for $\{1,\dotsc,n\}$. It is defined by the formula:
\[
\langle \sigma_{b_i}\sigma_{b_j} \rangle := 
\sum_{\sigma\in \{-1,1\}^{V}}\mathbb{P}(\sigma)\sigma_{b_i}\sigma_{b_j}.
\]

The $Z$\emph{-invariant Ising model} was introduced by Baxter in~\cite{baxter}
and consists in choosing coupling constants such that the boundary spin
correlations depend only on the region $R$ and not on the tiling $\mathbb{T}$.
This constraint yields a set of equations called the \emph{Yang-Baxter
equations}, see Section~\ref{section_invariant} for more details. A solution to
these equations can be parametrized by rhombus angles $(\theta_e)_{e\in E}$
naturally assigned to edges of $G_{\mathbb{T}}$ and by an \emph{elliptic
modulus} $k\in\C$ such that $k^2\in (-\infty,1)$ as follows:

\begin{equation*} j_e := \frac{1}{2}\ln\left(\frac{1+\mathsf{sn}(\theta_e|k)}{\mathsf{cn}(\theta_e|k)}\right),
\end{equation*}
where $\mathsf{sn}$ and $\mathsf{cn}$ are rescaled versions of the classical
\emph{Jacobi elliptic functions} $\sn$ and $\dn$ (see
\cite[Chapter~16]{abramowitz1965handbook},~\cite{lawden} and the appendix for
more details on these functions) defined by, for all $t$ in $\C$, 
\begin{equation} \mathsf{sn}(t)=\sn\left(\frac{2K(k)}{\pi}t\Bigl|\Bigr.k\right) \quad \text{ and }\quad  \mathsf{cn}(t)=\cn\left(\frac{2K(k)}{\pi}t\Bigl|\Bigr.k\right)
\label{rescaled}
\end{equation}
where $K(k)=\int_{0}^{\frac{\pi}{2}}\frac{dx}{\sqrt{1-k^2\sin^2(x)}}$ is the
\emph{complete elliptic integral of the first kind}.

When $k=0$, the functions $\mathsf{sn}$ and $\mathsf{cn}$ become the usual
trigonometric functions $\sin$ and $\cos$. The model then corresponds to the
\emph{critical} $Z$-invariant Ising model, whose criticality is proved in
\cite{li2012critical},~\cite{cimasoni2013critical} and \cite{lis2014phase}. In
the critical case, Galashin~\cite{galashin2022formula} proves a formula,
depending only on $R$, for the boundary spin correlations. Our main contribution
is a generalization of this result to the whole $Z$-invariant regime.


\begin{figure}[ht!]
\centering
\def\svgwidth{9cm}
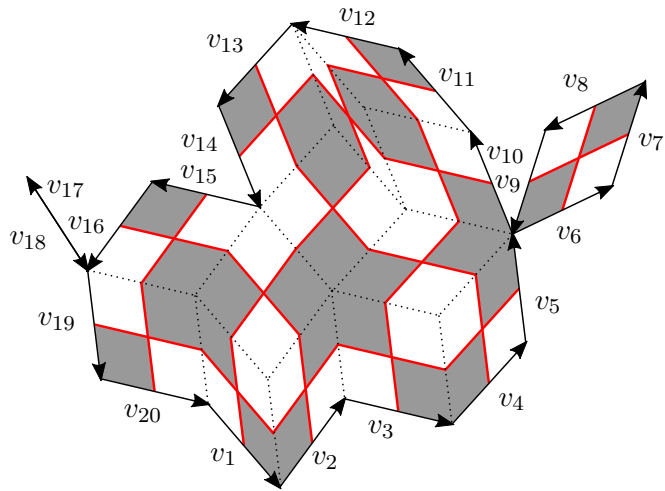
\caption{A region with a rhombus tiling (in dotted lines) and the traintracks (in red) associated to it}
\label{region_1}
\end{figure}

\begin{figure}[ht!]
\centering
\def\svgwidth{9cm}
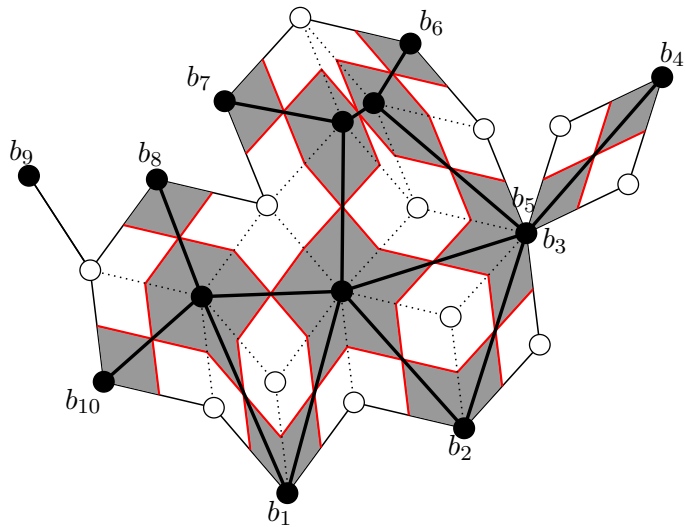
\caption{Graph associated to a region and a traintrack arrangement.}
\label{regions_graph}
\end{figure}


Let us fix the setting to formulate the statement of the main theorem. Let $R$
be a polygonal region whose boundary is a closed polygonal chain composed of
$2n$ unit vectors $v_1,\dotsc,v_{2n}$ listed in counterclockwise order. Let
$\mathbb{T}$ be a rhombus tiling of $R$ and let $\mathscr{T}$ be the set of
lines (called \emph{traintracks}, see~\cite{kenyon2005rhombic}) which are
obtained by connecting the midpoints of the opposite edges of each rhombus. For
all $j\in [n]$, the midpoint of the unit vector $v_j$ is linked to the midpoint
of a unit vector $v_k$ for some $k\in[n]$  by one of these lines. These links
can be recorded with a fixed-point-free involution $\tau:=\tau_R$ of $[2n]$
defined as follows:
for all $j\in [2n]$, let $k\in[2n]$ is the unique index such that the midpoint
of $v_j$ is linked to the midpoint of $v_k$. We then set $\tau(j)=k$ and
$\tau(k)=j$. Note that we have $v_j=-v_k$ when $\tau(j)=k$. Furthermore
$\tau$ depends only on the region $R$ and not on the rhombus tiling.

For technical reasons, we first consider regions for which there are no integers
$1\leq i<j<k<l\leq 2n$ such that $v_i=-v_j=v_k=-v_l$. Such regions are said to
be \emph{non-alternating}. The case of alternating regions is considered in
Section~\ref{alternating_regions}. We choose angles
$\alpha_1,\dotsc,\alpha_{2n}$ in a particular way, and set, for all $j$ in
$[2n]$, $v_j:=e^{i2\alpha_j}$, see Definition~\ref{tau_shape} for a precise
definition of these angles.

Now we define the function from which spin boundary correlations can be
computed. Let $k\in\C$ such that $k ^2\in (-\infty,1)$ and let
$k':=\sqrt{1-k^2}$. To a region $R$, associate a function
$\gamma_R:\R\rightarrow \R^{2n}$ with coordinates
$\gamma_R(t) = (\gamma_1(t),\ldots,\gamma_{2n}(t))$ where
\begin{equation}
 \forall t\in\R, \quad \forall p \in [2n], \quad \gamma_p(t)=(-1)^{|J_p\cap[p]|}\Pi_p(t)\prod_{j\in J_p}\mathsf{sn}(t-\alpha_j)
\label{gamma_1}
\end{equation}
with
\begin{equation}
\forall p \in  [2n], \quad J_p := \{j\in[2n] : (p,j,\tau(j)) \text{ are in cyclic order}\},
\label{Jp}
\end{equation}
\[ \Pi_p(t) := \prod_{j=1}^{2\lfloor\frac{p}{2}\rfloor} \frac{\mathsf{dn}(t-\alpha_j)}{\sqrt{k'}}, \]
and $\mathsf{dn}$ is a rescaled version of the Jacobi elliptic function $\dn$ as
in Equation~\eqref{rescaled}.

%
%

As explained earlier, the spin boundary correlations of the $Z$-invariant Ising
model depend only on the region $R$. Thus denote by $M_R$ the \emph{boundary
correlations matrix} defined by, for all $i,j$ in $[n]$, $(M_R)_{ij}=\langle
\sigma_{b_i}\sigma_{b_j} \rangle$. Following Galashin and
Pylyavskyy~\cite{galashin2020ising}, we introduce a modified version of $M_R$.
Let $\tilde{M}_R=(\tilde{m}_{i,j})$ be the $n\times 2n$ matrix defined in the
following way: for $i=j$, set $\tilde{m}_{i,2i-1}=\tilde{m}_{i,2i}=m_{i,i}=1$.
For $i\neq j$, set
\[ \tilde{m}_{i,2j-1}=-\tilde{m}_{i,2j}=(-1)^{i+j+\mathbbm{1}_{\{i<j\}}}m_{i,j}. \] 

\begin{exmp}
For instance, for $n=3$ the matrix $\tilde{M}$ is equal to
\[ \begin{pmatrix} 1 & 1 & m_{12} & -m_{12} & -m_{13} & m_{13} \\
- m_{12} & m_{12} & 1 & 1 & m_{23} & -m_{23} \\
 m_{13} & -m_{13} & -m_{23} & m_{23} & 1 & 1 \end{pmatrix}. \]
\end{exmp}

The row spawn of $\tilde{M}_R$ is called the \emph{doubling map} and is denoted
by $\phi(M_R)$. We can now state the main result of this paper that expresses
the doubling map in terms of the function $\gamma_R$ introduced above. We write
$\Span(\gamma_R)$ for the linear subspace of $\R^{2n}$ containing finite linear
combinations of elements $\gamma_R(t)$ where $t\in\R$.

\begin{teo}
Let $R$ be a non-alternating region composed of $2n$ unit vectors. Then
$\Span(\gamma_R)$ is of dimension $n$ and $\Span(\gamma_R)=\phi(M_R)$.
\label{main_teo}
\end{teo}
There is an analogous result for alternating regions, see Section
\ref{alternating_regions}. For regions with distinct angles $(\alpha_i)_{i\in
[2n]}$, a basis for $\Span(\gamma_R)$ is obtained in a simple way described by
the next proposition.

\begin{prop}
Let $R$ be a region composed of $2n$ unit vectors with angles
$\alpha_1,\dotsc,\alpha_{2n}$ such that $\alpha_i\neq\alpha_j$ for all $i\neq
j$. Write $J_1\cup\{1\}:=\{j_1,\dotsc,j_n\}$. Then
$\{\gamma_R(\alpha_{j_1}),\dotsc,\gamma_R(\alpha_{j_n})\}$ form a basis of
$\Span(\gamma_R)$.
\label{prop_basis}
\end{prop}

The next proposition complements Theorem~\ref{main_teo} by giving a practical
way of computing correlations knowing $\phi(M_R)$. Define the $2n\times n$
matrix $K_n$ by
\[ K_n=\frac{1}{2}\begin{pmatrix} 1 & 0 & \cdots & 0 \\
1 & 0 & \cdots & 0 \\
0 & 1 & \cdots & 0 \\
0 & 1 & \cdots & 0 \\
\vdots & \vdots & \ddots & \vdots \\
0 & 0 & \cdots & 1 \\
0 & 0 & \cdots & 1 \end{pmatrix}.\]
\begin{prop}[Galashin~\cite{galashin2022formula}]
Let $A$ be a $n\times 2n$ matrix whose row span is equal to $\phi(M_R)$. Then
the matrix $AK_n$ is invertible and it satifies
\[ (AK_n)^{-1}A=\tilde{M}_R. \]
\label{prop_K_n}
\end{prop}
So now give an example to illustrate computations.

\begin{exmp}
Let $n=2$ and $R:=R_2$ be the square region, see Figure~\ref{square}. We have
\[ \alpha_1=0, \;\alpha_2=\frac{\pi}{4}, \;\alpha_3=\frac{\pi}{2} \text{ and } \alpha_4=\frac{3\pi}{4}. \]
We also have
\[ \tau(1)=3, \; \tau(2)=4, \; \tau(3)=1 \text{ and } \tau(4)=2,\]
so we get
\[ J_1=\{2\}, J_2 = \{3\}, J_3=\{4\} \text{ and } J_4=\{1\}.\]

\begin{figure}[h!]
\centering
\def\svgwidth{5cm}
\begingroup%
  \makeatletter%
  \providecommand\color[2][]{%
    \errmessage{(Inkscape) Color is used for the text in Inkscape, but the package 'color.sty' is not loaded}%
    \renewcommand\color[2][]{}%
  }%
  \providecommand\transparent[1]{%
    \errmessage{(Inkscape) Transparency is used (non-zero) for the text in Inkscape, but the package 'transparent.sty' is not loaded}%
    \renewcommand\transparent[1]{}%
  }%
  \providecommand\rotatebox[2]{#2}%
  \newcommand*\fsize{\dimexpr\f@size pt\relax}%
  \newcommand*\lineheight[1]{\fontsize{\fsize}{#1\fsize}\selectfont}%
  \ifx\svgwidth\undefined%
    \setlength{\unitlength}{44.66663919bp}%
    \ifx\svgscale\undefined%
      \relax%
    \else%
      \setlength{\unitlength}{\unitlength * \real{\svgscale}}%
    \fi%
  \else%
    \setlength{\unitlength}{\svgwidth}%
  \fi%
  \global\let\svgwidth\undefined%
  \global\let\svgscale\undefined%
  \makeatother%
  \begin{picture}(1,0.80381238)%
    \lineheight{1}%
    \setlength\tabcolsep{0pt}%
    \put(0,0){\includegraphics[width=\unitlength,page=1]{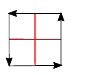}}%
    \put(0.34828328,0.02375431){\color[rgb]{0,0,0}\makebox(0,0)[lt]{\lineheight{1.25}\smash{\begin{tabular}[t]{l}$v_1$\end{tabular}}}}%
    \put(0.68598767,0.36429938){\color[rgb]{0,0,0}\makebox(0,0)[lt]{\lineheight{1.25}\smash{\begin{tabular}[t]{l}$v_2$\end{tabular}}}}%
    \put(0.34467082,0.70463916){\color[rgb]{0,0,0}\makebox(0,0)[lt]{\lineheight{1.25}\smash{\begin{tabular}[t]{l}$v_3$\end{tabular}}}}%
    \put(0,0){\includegraphics[width=\unitlength,page=2]{square.pdf}}%
    \put(-0.00836073,0.38338861){\color[rgb]{0,0,0}\makebox(0,0)[lt]{\lineheight{1.25}\smash{\begin{tabular}[t]{l}$v_4$\end{tabular}}}}%
    \put(0.70171864,0.02924296){\color[rgb]{0,0,0}\makebox(0,0)[lt]{\lineheight{1.25}\smash{\begin{tabular}[t]{l}$b_1$\end{tabular}}}}%
    \put(0.00667689,0.7272871){\color[rgb]{0,0,0}\makebox(0,0)[lt]{\lineheight{1.25}\smash{\begin{tabular}[t]{l}$b_2$\end{tabular}}}}%
  \end{picture}%
\endgroup%

\caption{The region $R_2$}
\label{square}
\end{figure}

Let $A$ be the matrix with rows $-\gamma_R(0)$ and $-\gamma_R(\frac{\pi}{4})$ that is
\begin{align*} A & = \begin{pmatrix}
\mathsf{sn}{\frac{\pi}{4}} & \mathsf{sn}(\frac{\pi}{2})\mathsf{dn}(\frac{\pi}{4})\frac{1}{k'} & \mathsf{sn}(\frac{3\pi}{4})\mathsf{dn}(\frac{\pi}{4})\frac{1}{k'} & 0 \\
0 & \mathsf{sn}(\frac{\pi}{4})\mathsf{dn}(\frac{\pi}{4})\frac{1}{k'} & \mathsf{sn}(\frac{\pi}{2})\mathsf{dn}(\frac{\pi}{4})\frac{1}{k'} & \mathsf{sn}(\frac{\pi}{4}) \\
\end{pmatrix} \\
& = \begin{pmatrix}
\frac{1}{\sqrt{1+k'}} & \frac{1}{\sqrt{k'}} & \frac{1}{\sqrt{k'(1+k')}} & 0 \\
0 & \frac{1}{\sqrt{k'(1+k')}} & \frac{1}{\sqrt{k'}} & \frac{1}{\sqrt{1+k'}} \\
\end{pmatrix}
\end{align*}
using elliptic functions identities~\eqref{value_sn} and~\eqref{value_dn} for
the last equality. Thus
\[ AK_n = \frac{1}{2}\begin{pmatrix}
\frac{1}{\sqrt{1+k'}}+\frac{1}{\sqrt{k'}} & \frac{1}{\sqrt{k'(1+k')}} \\
\frac{1}{\sqrt{k'(1+k')}} &  \frac{1}{\sqrt{k'}}+\frac{1}{\sqrt{1+k'}} \\
\end{pmatrix}
\]
Then we compute
\[ (AK_n)^{-1} = \frac{1}{\frac{1}{1+k'}+\frac{1}{\sqrt{k'(1+k')}}}\begin{pmatrix} \frac{1}{\sqrt{1+k'}}+\frac{1}{\sqrt{k'}} & -\frac{1}{\sqrt{k'(1+k')}} \\
-\frac{1}{\sqrt{k'(1+k')}} & \frac{1}{\sqrt{k'}}+\frac{1}{\sqrt{1+k'}} \\
\end{pmatrix} \]
and finally we obtain
\[ (AK_n)^{-1}A = \begin{pmatrix}
1  & 1 & \frac{1}{\sqrt{k'}+\sqrt{1+k'}} & -\frac{1}{\sqrt{k'}+\sqrt{1+k'}} \\
-\frac{1}{\sqrt{k'}+\sqrt{1+k'}} & \frac{1}{\sqrt{k'}+\sqrt{1+k'}} & 1 & 1 \\
\end{pmatrix}.
\]
Therefore we deduce that
\[ \langle \sigma_{b_1}\sigma_{b_2} \rangle = \frac{1}{\sqrt{k'}+\sqrt{1+k'}}. \]
In this simple case, the same value can be found by the following direct
computation. For the $R_2$ region, the constant $j_e$ associated to the single
edge $e$ of the graph is equal to
$\frac{1}{2}\ln\left(\frac{1+\mathsf{sn}(\frac{\pi}{4})}{\mathsf{\cn}(\frac{\pi}{4})}\right)$
by Equation~\eqref{Z_invariant_weights}. Then

\begin{align*}
\langle \sigma_{b_1}\sigma_{b_2} \rangle & = \frac{2e^J-2e^{-J}}{2e^J+2e^{-J}} \\
	& = \left(\sqrt{\frac{1+\mathsf{sn}(\frac{\pi}{4})}{\mathsf{\cn}(\frac{\pi}{4})}}-\sqrt{\frac{\mathsf{\cn}(\frac{\pi}{4})}{1+\mathsf{sn}(\frac{\pi}{4})}}\right)\left(\sqrt{\frac{1+\mathsf{sn}(\frac{\pi}{4})}{\mathsf{\cn}(\frac{\pi}{4})}}+\sqrt{\frac{\mathsf{\cn}(\frac{\pi}{4})}{1+\mathsf{sn}(\frac{\pi}{4})}}\right)^{-1} \\
	& = \frac{1+\mathsf{sn}(\frac{\pi}{4})-\mathsf{\cn}(\frac{\pi}{4})}{1+\mathsf{sn}(\frac{\pi}{4})+\mathsf{\cn}(\frac{\pi}{4})} \\
	& = \frac{1+\sqrt{1+k'}-\sqrt{k'}}{1+\sqrt{1+k'}+\sqrt{k'}} \quad \text{ using the identities~\eqref{value_sn} and~\eqref{value_cn}} \\
	& = \frac{1+\sqrt{1+k'}-\sqrt{k'}}{(1+\sqrt{1+k'}-\sqrt{k'})(\sqrt{1+k'}+\sqrt{k'})} \\
	& = \frac{1}{\sqrt{k'}+\sqrt{1+k'}}.
\end{align*}

\end{exmp}

\begin{rem}
In general, the method used to compute correlations in the Ising model goes
through the \emph{Kasteleyn matrix}~\cite{fisher, Kasteleyn, fisher2} or the
\emph{Kac-Ward matrix}~\cite{kacwardref,lis2016short} associated to it. But
these matrices depend on the rhombus tiling with a number of rows and columns
which is of order the number of tiles that is $O(n^2)$. Here the matrix $AK_n$
is only of size $n\times n$ but the coefficients are more complex.
\end{rem}

\paragraph{Outline of the paper.} In Section~\ref{section_preli}, we introduce
the model and the notations. Then we prove the main theorem of the paper in
Section~\ref{section_main_teo}. We discuss duality and near-critical case in the
rest.

\paragraph{Acknowledgements.} I am very grateful to my supervisors C\'edric
Boutillier and B\'eatrice de Tili\`ere.

\section{Preliminaries}
\label{section_preli}

We first define general polygonal regions and graphs arising from them since the
proof of the main result goes through degenerate regions. Then we introduce the
$Z$-invariant Ising model and we explain how the doubling map simplifies the
computations of boundary correlations.

\begin{rem}
We follow closely the notations found in~\cite{galashin2022formula} with a few
exceptions: the angles of a region are written $\alpha$ instead of $\theta$, we
use the word \emph{traintracks} instead of \emph{pseudolines} and the semi-angle
$\theta_e$ associated to a rhombus is equal to $\frac{\pi}{2} $ minus the one of
Galashin.
\end{rem}

\subsection{Regions and graphs}

Let $n\geq 1$ be an integer. Let $\tau:[2n]\rightarrow [2n]$ be a
fixed-point-free involution and let
$\bm{\alpha}=(\alpha_1,\dotsc,\alpha_{2n})\in\R^{2n}$ be a sequence of $2n$
angles satisfying the two following conditions:

\begin{equation}
\forall j < \tau(j), \quad \alpha_{\tau(j)}=\alpha_j+\frac{\pi}{2},
\label{cond_1}
\end{equation} 
and
\begin{equation}
\forall j<k<\tau(j)<\tau(k), \quad \alpha_j<\alpha_k<\alpha_{\tau(j)}<\alpha_{\tau(k)}. 
\label{cond_2}
\end{equation} 

%

\begin{defi}
  A sequence of angles $\bm{\alpha}$ satisfying the conditions~\eqref{cond_1}
  and~\eqref{cond_2}
  is called a $\tau$\emph{-shape}. A \emph{region} $R$ is a pair
  $(\tau,\bm{\alpha})$ with $\tau$ a fixed-point-free involution and $\bm{\alpha}$
  a $\tau$-shape.
  \label{tau_shape}
\end{defi}
Define $2n$ unit vectors $v_1,\dotsc,v_{2n}$ by
\[ \forall j\in [2n], \quad v_j:=e^{i2\alpha_j}.\]
Note that the number $2$ in the definition is there to simplify computations in
the proof of the main result. A region $R$ can be drawn in the plane as a closed
polygonal chain with sides $v_1,\dotsc,v_{2n}$ (see Figure~\ref{region_1}). Note
that in general this polygonal chain may intersect itself. If it is
non-self-intersecting, we say that the region $R$ is \emph{simple}.

Let $d_1,\dotsc,d_{2n}$ be $2n$ points drawn counterclockwise on a disk. A
\emph{traintrack arrangement} $\mathscr{T}$ of $(d_1,\dotsc,d_{2n})$ is a set of
$n$ lines, called \emph{traintracks}, which pair these $2n$ points such that
there is no self-intersection, any two lines intersect at most one, intersection
points are in the interior of the disk and no three lines intersect at one
point. The pairing induced by a traintrack arrangement $\mathscr{T}$ defines a
fixed-point-free involution $\tau:=\tau_{\mathscr{T}}$ where for all
$j,k\in[2n]$ we have $\tau(k)=j$ and $\tau(j)=k$ if and only if $d_j$ and $d_k$
are paired. Conversely, for every fixed-point-free involution $\tau$, there
exists a traintracks arrangement $\mathscr{T}$ such that
$\tau=\tau_{\mathscr{T}}$: draw straight lines pairing $d_j$ to $d_{\tau(j)}$
for all $j\in [n]$ for instance (avoiding three lines intersection by curving
the lines if needed).

Let $R=(\tau,\bm{\alpha})$ be a region and let $\mathscr{T}$ be a traintrack
arrangement such that $\tau=\tau_{\mathscr{T}}$. Note that $\mathscr{T}$ divides
the region $R$ into faces that can be colored in black and white in a bipartite
way such that, for all $j\in[n]$, the face next to the arc of the disk
connecting $d_{2j-1}$ and $d_{2j}$ is black (see Figure~\ref{region_2}). 

\begin{figure}[ht!]
\centering
\def\svgwidth{8cm}
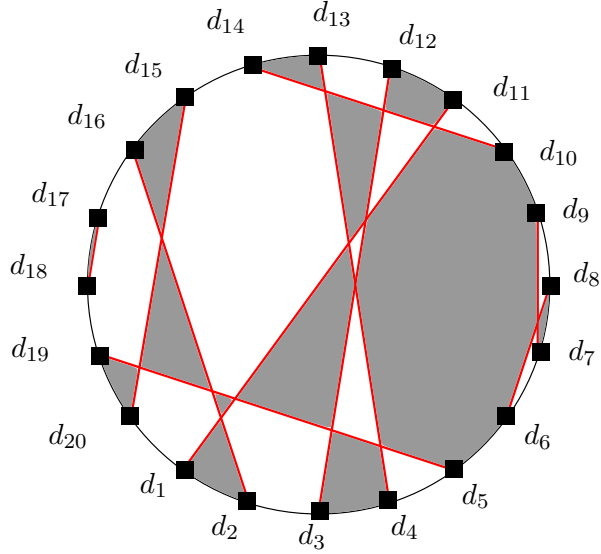
\caption{Another way of representing the region $R$ of Figure~\ref{region_1} (with a different choice of traintrack arrangement)}
\label{region_2}
\end{figure}

\begin{defi}
Let $G^{\bullet}_{\mathscr{T}}=(V,E)$ be the graph whose vertices are the black
faces and whose edges join black faces that are incident to a common
intersection point of traintracks (see Figure~\ref{regions_graph}). Let
$G^{\circ}_{\mathscr{T}}$ be the \emph{dual graph} of
$G^{\bullet}_{\mathscr{T}}$, \emph{i.e.}, the graph whose vertices are the white
faces and whose edges join white faces that are incident to a common
intersection point of traintracks. Let $G^{\lozenge}_{\mathscr{T}}$ be the
\emph{diamond graph} of $G^{\bullet}_{\mathscr{T}}$ that is the bipartite graph
whose vertices are the white faces and the black faces and whose edges join
white faces and black faces that are adjacent.
\label{def_graphs}
\end{defi}

\begin{rem}
The choice of coloring is arbitrary and switching the role of black and white is
equivalent to switching the definition of $G^{\bullet}_{\mathscr{T}}$ and its
dual. If one wants to choose the other way of coloring, it suffices to relabel
$(v_1,\dotsc,v_{2n})$ into $(v_2,\dotsc,v_{2n},v_1)$ to get back to our setting.
This choice of labeling is related to the duality of the Ising model.
For a more extensive discussion about this duality, see
Section~\ref{section_duality}.
\end{rem}

There are $n$ boundary black faces that we denote by $b_1,\dotsc,b_n$ where
$b_j$ is the black face adjacent to the vertices $d_{2j-1}$ and $d_{2j}$ for all
$j\in[n]$. Note that we may have several boundary points
$b_{j_1},\dotsc,b_{j_k}$ corresponding to the same black face, as $b_3$ and
$b_5$ in Figure~\ref{region_2}. In this case, we consider these points as if
they were ``contracted'' into a single vertex and we set $\langle
\sigma_{b_{j_r}}\sigma_{b_{j_s}}\rangle :=1$ for all $r,s\in [k]$,
see~\cite[Definition 6.1]{galashin2020ising} for details.

Now for each edge $e\in E$, we define an angle $\theta_e$ as follows. Each edge
$e$ corresponds to the intersection point of a traintrack connecting $d_j$ to
$d_{\tau(j)}$ and another traintrack connecting $d_k$ to $d_{\tau(k)}$ for some
$j<k<\tau(j)<\tau(k)$. We set either $\theta_e:=\alpha_k-\alpha_j$ or
$\theta_e:=\alpha_{\tau(j)}-\alpha_k$ depending on how the two black faces
linked by $e$ are located relative to $j,k,\tau(j),\tau(k)$, see Figure
\ref{def_theta}. For simple regions, there is a \emph{rhombus tiling} associated
to a traintrack arrangement (see~\cite{kenyon2005rhombic}) and then $\theta_e$
corresponds to the semi-angle adjacent to $e$ of the rhombus containing $e$, see
Figure~\ref{rhombus_theta}.

\begin{figure}[ht!]
\centering
\def\svgwidth{4.8cm}
\subcaptionbox{Angle $\theta_e$ in the rhombus containing $e$ \label{rhombus_theta}}
{
\begingroup%
  \makeatletter%
  \providecommand\color[2][]{%
    \errmessage{(Inkscape) Color is used for the text in Inkscape, but the package 'color.sty' is not loaded}%
    \renewcommand\color[2][]{}%
  }%
  \providecommand\transparent[1]{%
    \errmessage{(Inkscape) Transparency is used (non-zero) for the text in Inkscape, but the package 'transparent.sty' is not loaded}%
    \renewcommand\transparent[1]{}%
  }%
  \providecommand\rotatebox[2]{#2}%
  \newcommand*\fsize{\dimexpr\f@size pt\relax}%
  \newcommand*\lineheight[1]{\fontsize{\fsize}{#1\fsize}\selectfont}%
  \ifx\svgwidth\undefined%
    \setlength{\unitlength}{47.67281811bp}%
    \ifx\svgscale\undefined%
      \relax%
    \else%
      \setlength{\unitlength}{\unitlength * \real{\svgscale}}%
    \fi%
  \else%
    \setlength{\unitlength}{\svgwidth}%
  \fi%
  \global\let\svgwidth\undefined%
  \global\let\svgscale\undefined%
  \makeatother%
  \begin{picture}(1,0.74977479)%
    \lineheight{1}%
    \setlength\tabcolsep{0pt}%
    \put(0,0){\includegraphics[width=\unitlength,page=1]{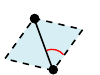}}%
    \put(0.51090614,0.30054186){\color[rgb]{0,0,0}\makebox(0,0)[lt]{\lineheight{1.25}\smash{\begin{tabular}[t]{l}$\theta_e$\end{tabular}}}}%
    \put(0,0){\includegraphics[width=\unitlength,page=2]{rhombus_theta.pdf}}%
    \put(0.25520212,0.64702744){\color[rgb]{0,0,0}\makebox(0,0)[lt]{\lineheight{1.25}\smash{\begin{tabular}[t]{l}$e$\end{tabular}}}}%
  \end{picture}%
\endgroup%
}
\def\svgwidth{10cm}
\subcaptionbox{The two cases of the definition of $\theta_e$: $\theta_e=\alpha_k-\alpha_j$ on the left and $\theta_e=\alpha_{\tau(j)}-\alpha_k$ on the right. \label{def_theta}}
{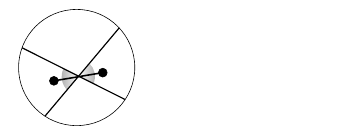}
\caption{Rhombus definition and abstract definition of the angle 
$\theta_e$}
\end{figure}

It will be convenient to extend the definition of $\tau$ and $\bm{\alpha}$.
Given a fixed-point-free involution $\tau:[2n]\rightarrow [2n]$, define
$\tilde{\tau}:\Z\rightarrow \Z$ to be a lift of $\tau$ satisfying the following
conditions:
\[ \forall k \in \Z, \quad \tilde{\tau}(k+2n)=\tilde{\tau}(k)+2n, \]
\[ \forall k \in \Z,\quad k<\tilde{\tau}(k)<k+2n, \]
and
\[ \forall k \in [2n], \quad \tilde{\tau}(k)\equiv \tau(k) (\text{mod} \; 2n) .\]
Similarly, given a sequence of angles $\bm{\alpha}=(\alpha_j)_{j\in [2n]}$,
define $\tilde{\bm{\alpha}}=(\tilde{\alpha}_j)_{j \in \Z}$ by
\[ \forall k \in [2n], \quad \tilde{\alpha}_k=\alpha_k \]
and
\[ \forall k \in \Z, \quad \tilde{\alpha}_{k+2n}=\tilde{\alpha}_k+\pi .\]
Finally, for $k\in \Z$, set $\tilde{v}_k=\exp(2i\tilde{\alpha}_k)$. This notation allows us to have more symmetric formulas. Define, for all $p\in \Z$,
\begin{equation} \tilde{J}_p:=\{\tilde{\tau}(j)\;|\;j\in\Z \text{ such that } j<p \text{ and } \tilde{\tau}(j)>p \}.
\label{tilde_Jp}
\end{equation}
The set $\tilde{J}_p$ coincides modulo $2n$ with the set $J_p$ defined in~\eqref{Jp} for $p\in [2n]$. Equality~\eqref{gamma_1} can then be rewritten, for $p\in [2n]$,
\begin{equation}
 \gamma_p(t)=\Pi_p(t)\prod_{j\in \tilde{J}_p}\mathsf{sn}(t-\tilde{\alpha}_j)
 \label{tilde_gamma}
\end{equation}
using the fact that $\mathsf{sn}(u+\pi)=-\mathsf{sn}(u)$ for all $u\in \C$. We
extend the definition of $\gamma_p(t)$ for all $p\in \Z$ using
Equality~\eqref{tilde_gamma} as a definition. Note that we have
$\gamma_{p+2n}(t)=(-1)^{n-1}\gamma_p(t)$ for all $p\in \Z$. From now on, we
switch freely between the notation with or without tilde.

\subsection{The Z-invariant Ising Model}\label{section_invariant}
Let $R$ be a region and let $\mathscr{T}$ be a traintrack arrangement. Consider
the \emph{Ising model} on $G_{\mathscr{T}}^{\bullet}$ with \emph{coupling
constants} $(j_e)_{e\in E}$. Following Baxter~\cite{baxter}, the
$(j_e)_{e\in E}$ are chosen in a way that boundary correlations depend only on
the region $R$ and not on the traintrack arrangement. More precisely, we call
\emph{star} a
vertex of degree $3$ and we define a \emph{star-triangle move} to be a local
tranformation of a triangle in the graph $G_{\mathscr{T}}^{\bullet}$ into a star
which
gives a new graph $G_{\mathscr{T}'}^{\bullet}$ (this can be understood as the
movement of
a traintrack over an intersection point, see Figure~\ref{star_triangle}).
\begin{figure}[ht!]
\centering
\def\svgwidth{10cm}
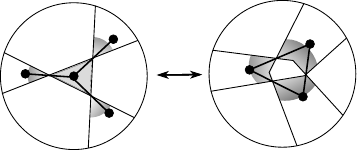
\caption{Star-triangle move}
\label{star_triangle}
\end{figure}
The Ising model is said to be $Z$\emph{-invariant} if, when decomposing the
partition function according to the possible spin configurations at the three
vertices of a star/triangle, it only changes by a constant when performing the
star-triangle move, this constant being independent of the choice of the spins
at the three vertices. These constraints can be written as a set of equations
called the \emph{Yang-Baxter equations} for the Ising model. A parametrization
of the solutions of these equations was found by Baxter in~\cite{baxter} and
then rewritten in~\cite{boutillier2019z}. Let $k$ be a complex number such that
$k^2\in(-\infty,1)$, $k$ is known as the \emph{elliptic modulus}. For every edge
$e$ of $G_{\mathscr{T}}^{\bullet}$, set
\begin{equation} j_e := \frac{1}{2}\ln\left(\frac{1+\mathsf{sn}(\theta_e|k)}{\mathsf{cn}(\theta_e|k)}\right).
\label{Z_invariant_weights}
\end{equation}
Kenyon showed in~\cite{Kenyon2005TilingAP} that any graph
$G_{\mathscr{T}}^{\bullet}$, where $\mathscr{T}$ is a traintrack arrangement,
can be turned into a graph
$G_{\mathscr{T}'}^{\bullet}$, with $\mathscr{T}'$ another traintrack arrangement, applying
only star-triangle moves. This implies that boundary correlations depend only on
the region and not on the choice of the traintrack arrangement.

\subsection{Removing a crossing}

Following Galashin~\cite{galashin2022formula}, let us define formally
the notion of crossing as well as that of descent, which corresponds to a
crossing adjacent to a boundary.

\begin{defi}
Let $R=(\tau,\bm{\alpha})$ be a region. We say that $R$ has a \emph{crossing} if
there exists $j,k \in [2n]$ such that $j<k<\tau(j)<\tau(k)$.
An index $j\in [2n]$ satisfying $j<j+1<\tilde{\tau}(j)<\tilde{\tau}(j+1)$ is called a $\tau$\emph{-descent}.
\end{defi}

\begin{rem}\label{no_tau_descent}
  Note that in~\cite{galashin2022formula}, the author claims that a
  region with crossings necessarily has a $\tau$-descent. But this is not always
  the case, see Figure~\ref{no_descent} for an example. We thus consider this
  possibility in the proof of the main theorem.
\end{rem}

\begin{figure}[ht]
\centering
\def\svgwidth{6cm}
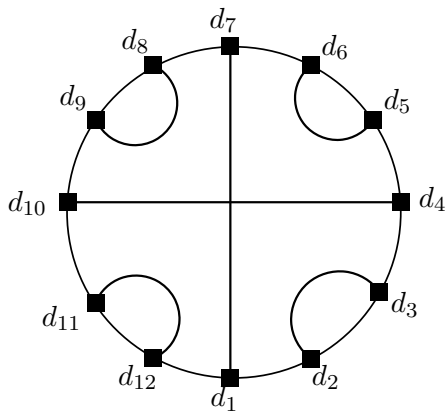
\caption{A region with a crossing but without a $\tau$-descent}\label{no_descent}
\end{figure}

\begin{defi}\label{region_remove}
  Let $R=(\tau,\bm{\alpha})$ be a region with a $\tau$-descent $j\in [2n]$. Denote
  by $t_j$ the transposition of $[2n]$ that swaps $j$ and $j+1$. Define a new
  region $R':=R\cdot t_j:=(\tau',\bm{\alpha}')$ by $\tau'=\tau\circ t_j$ and, for
  all $l\in[2n]$, $\alpha'_{l}=\alpha_{t_j(l)}$.
  One can check that $R'$ is actually a region. Note also that if $R$ is
  non-alternating, then $R'$ is also non-alternating.
\end{defi}

\begin{figure}[ht]
\centering
\def\svgwidth{6cm}
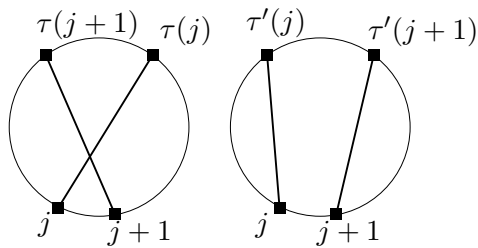
\caption{A region $R$ with a $\tau$-descent $j$ (left) and the region $R'=R\cdot t_j$ (right)}
\label{remove_cross}
\end{figure}

We now explain how the linear space $\phi(M_R)$ is modified when we remove a
$\tau$-descent. Let $R=(\tau,\bm{\alpha})$ be a region with a $\tau$-descent
$j\in [2n]$. There is an edge $e$ associated to the intersection point between
the traintrack incident to $d_j$ and the traintrack incident to $d_{j+1}$. Set
$s_e:= (\ch(2j_e))^{-1}$ and $c_e:=\tanh(2j_e)$. We can express these quantities
as functions of $\alpha_j$ and $\alpha_{j+1}$. To gain in clarity, define the
functions $\mathsf{sn}^*$, $\mathsf{cn}^*$ and $\mathsf{dn}^*$ as follows. Let
$k^*:=i\frac{k}{k'}$ be the \emph{dual parameter} of $k$ and for all $t\in\C$
set
\[ \mathsf{sn}^*(t)=\sn\left(\frac{2K(k^*)}{\pi}t\Bigl|\Bigr.k^*\right), \quad  \mathsf{cn}^*(t)=\cn\left(\frac{2K(k^*)}{\pi}t\Bigl|\Bigr.k^*\right)\]
and
\[ \mathsf{dn}^*(t)=\dn\left(\frac{2K(k^*)}{\pi}t\Bigl|\Bigr.k^*\right). \]
Note that these rescaled functions are then $2\pi$-periodic.
\begin{itemize}
\item If $j$ is odd then $\theta_e=\frac{\pi}{2}-(\alpha_{j+1}-\alpha_j)$. Therefore,
\[
s_e  = \frac{1}{2}\left(\frac{1+\mathsf{sn}(\theta_e)}{\mathsf{cn}(\theta_e)}+\frac{\mathsf{cn}(\theta_e)}{1+\mathsf{sn}(\theta_e)}\right) 
	 = \mathsf{cn}(\theta_e) 
	 = \mathsf{cn}\left(\frac{\pi}{2}-(\alpha_{j+1}-\alpha_j) \right).
\]
By Table~\ref{table_formula} and the duality formula~\eqref{dual_sn} we thus get
\[ s_e = k'\mathsf{sd}\left(\alpha_{j+1}-\alpha_j\right) = \mathsf{sn}^*\left(\alpha_{j+1}-\alpha_j\right). \]
Similarly, using again Table~\ref{table_formula} and the other duality formula~\eqref{dual_cn}, we have
\[ c_e  = \mathsf{sn}\left(\frac{\pi}{2}-(\alpha_{j+1}-\alpha_j)\right) 
 = \mathsf{cd}\left(\alpha_{j+1}-\alpha_j\right) 
 = \mathsf{cn}^*\left(\alpha_{j+1}-\alpha_j\right).
\]

\item If $j$ is even then $\theta_e=\alpha_{j+1}-\alpha_j$. Therefore
\[ s_e = \mathsf{cn}\left(\alpha_{j+1}-\alpha_j\right) \; \text{ and } \; c_e = \mathsf{sn}\left(\alpha_{j+1}-\alpha_j\right). \]
\end{itemize}
 For all $j\in \Z$, set $c_j=\mathsf{cn}(\alpha_{j+1}-\alpha_j), s_j=\mathsf{sn}(\alpha_{j+1}-\alpha_j), c_j^*=\mathsf{cn}^*(\alpha_{j+1}-\alpha_j)$ and $s_j^*=\mathsf{sn}^*(\alpha_{j+1}-\alpha_j)$.

\begin{defi}
Define, for all $j\in \Z$, the $2\times 2$ matrix $B^{\bm{\alpha}}_j$ by

\[ B_j^{\bm{\alpha}} :=\left\{
    \begin{array}{ll}
        \begin{pmatrix}
1/c_j^* & s_j^*/c_j^* \\
s_j^*/c_j^* & 1/c_j^*
\end{pmatrix} & \text{if } j \text{ is odd,} \\
        \begin{pmatrix}
1/c_j & s_j/c_j \\
s_j/c_j & 1/c_j
\end{pmatrix} & \text{if } j \text{ is even.}
    \end{array}
\right. \]
For $1\leq j<2n$, define the matrix $g_j^{\bm{\alpha}}$ which coincides with the
$2n\times 2n$ identity matrix except for the $2\times 2$ block
$B^{\bm{\alpha}}_j$ which appears in rows and columns $j,j+1$. For $j=2n$, the
matrix $g_{2n}^{\bm{\alpha}}$ coincides with the $2n\times 2n$ identity matrix
except for the following four entries:
\[ (g_{2n}^{\bm{\alpha}})_{1,1}=(g_{2n}^{\bm{\alpha}})_{2n,2n}=\frac{1}{s_{2n}} \text{ and } (g_{2n}^{\bm{\alpha}})_{1,2n}=(g_{2n}^{\bm{\alpha}})_{2n,1}=(-1)^{n-1}\frac{c_{2n}}{s_{2n}}. \]
\label{g_p}
\end{defi}

Theorem~3.22 of Galashin and Pylyavskyy in~\cite{galashin2020ising} gives in a
more general setting an equality between the linear spaces $\phi(M_R)$ and
$\phi(M_{R'})$. In our particular case, it is read as follows. For a linear
subspace $V\subset \R^{2n}$ of dimension $n$ and for a matrix $g$ of size
$2n\times 2n$, define the linear subspace $V\cdot g$ as $\{v\cdot g : v\in V\}$
where $v\in V$ is seen as a row vector.

\begin{teo}[\cite{galashin2020ising}]
Let $R=(\tau,\bm{\alpha})$ be a region and $j\in[2n]$ be a $\tau$-descent. Then for $R':=R\cdot t_j$, we have
\[ \phi(M_R)= \phi(M_{R'})\cdot g_j^{\bm{\alpha}}. \]
\label{theo_phi}
\end{teo}

\begin{rem}
\leavevmode
\begin{itemize}
\item Note that for a $\tau$-descent $j\in [2n]$, the matrix $g_j^{\bm{\alpha}}$
  is well-defined. Indeed, we have $\tau(j)\neq j+1$. Then
  $\alpha_j<\alpha_{j+1}<\alpha_{j}+\frac{\pi}{2}$ by Condition~\eqref{cond_2}.
  Thus $c_j\neq 0$ and $c_j^*\neq 0$ which are the numbers appearing as
  denominators of the fractions defining the entries of $B_j^{\bm{\alpha}}$.
  In that case, $g_j^{\bm{\alpha}}$ is invertible and its determinant is equal
  to~1.
\item The critical $Z$-invariant Ising model corresponds to the elliptic modulus
  $k=0$. In this case, we have $k=k^*=0$ so there is no need to distinguish the
  parity of $j$ in the definition of $B_j^{\bm{\alpha}}$ and $g_j^{\bm{\alpha}}$
  since $c_j=c_j^*$ and $s_j=s_j^*$.  The computations of the next section are
  therefore much simpler in the critical $Z$-invariant case.
\end{itemize}
\end{rem}

\section{Explicit formula for the doubling map}
\label{section_main_teo}

\subsection{Main result}

The main result of this paper gives an explicit expression for $\phi(M_R)$.
Recall that, to a region $R$, we associate a curve $\gamma_R:\R\rightarrow
\R^{2n}$ with coordinates $\gamma_R(t) = (\gamma_1(t),\ldots,\gamma_{2n}(t))$
where
\begin{equation}
 \forall p \in [2n], \quad \gamma_p(t)=\Pi_p(t)\prod_{j\in \tilde{J}_p}\mathsf{sn}(t-\tilde{\alpha}_j)
 \label{gamma}
\end{equation}
with
\begin{equation} \Pi_p(t) := \prod_{j=1}^{2\lfloor\frac{p}{2}\rfloor} \frac{\mathsf{dn}(t-\tilde{\alpha}_j)}{\sqrt{k'}} = \prod_{j=1}^{\lfloor\frac{p}{2}\rfloor}\frac{\mathsf{dn}(t-\tilde{\alpha}_{2j})}{\mathsf{dn}(t-\tilde{\alpha}_{\tau(2j-1)})}.
\label{product}
\end{equation}

We make several remarks concerning Formula~\eqref{gamma}.

\begin{rem}\label{remark_formula}
\leavevmode
\begin{itemize}
\item The function $\gamma_R$ depends only on the region $R$ and not on the
  rhombus tiling.
\item By changing the choice of the black/white coloring which leads to
  Definition~\ref{def_graphs}, it implies a switch between ``odd'' and ``even''
  in Definition~\ref{g_p}. All results and computations remain valid by
  replacing in~\eqref{gamma} and~\eqref{product}, but also in the coming
  computations, the rescaled `primal' Jacobi elliptic function
  $\mathsf{sn},\mathsf{dn},\mathsf{cn}$ by their dual version
  $\mathsf{sn}^*,\mathsf{dn}^*,\mathsf{cn}^*$.
\end{itemize}
\end{rem}

The product $\Pi_p(t)$ can be seen as a product associated to a
path. Let $R$ be a region and let $\mathscr{T}$ be the set of $n$ train-tracks
associated to it with endpoints $d_1,\dotsc,d_{2n}$ and associated angles
$\alpha_1,\dotsc,\alpha_{2n}$. Let $G^{\lozenge}_{\mathscr{T}}$ be the diamond
graph of Definition~\ref{def_graphs}. Let $f$ and $f'$ be two vertices of
$G^{\lozenge}_{\mathscr{T}}$. To a path $\Gamma_{f,f'}$ in
$G^{\lozenge}_{\mathscr{T}}$ going from $f$ to $f'$ we associate the following
product:
\[ 
  \prod_{e=xy\in \Gamma_{f,f'}}\omega_{(x,y)}(t)
\]
where
\[
  \omega_{(x,y)}(t) := \frac{1}{\sqrt{k'}}\mathsf{dn}(t-\alpha_{j(x,y)}),
\]
and $j(x,y)$ is such that $d_{j(x,y)}$ is the endpoint of the train-track
crossed which is on the left of the oriented edge $(x,y)$. This
product does not depend on the choice of path between $f$ and $f'$. Indeed the
product of the multipliers on a path surrounding a rhombus is equal to $1$ using
Table~\ref{table_formula}. Let $f$ and $f'$ be two vertices of
$G^{\lozenge}_{\mathscr{T}}$ that are adjacent to the boundary of $R$. The path
that goes from $f$ to $f'$ going only through boundary vertices exists for any
rhombus tiling (see Figure~\ref{path}). Thus the product associated to it is the
same for any tiling. To a number $p\in[2n]$ we associate the white face $f_p$
touching the boundary between $p$ and $p+1$ if $p$ is even and we associate the
white face $f_p$ touching the boundary between $p-1$ and $p$ if $p$ is odd
(modulo $2n$). Then $\Pi_p(t)$ is equal to the product associated to the path
that goes from $f_p$ to $f_1$.

\begin{figure}[ht!]
  \centering
  \def\svgwidth{8cm}
  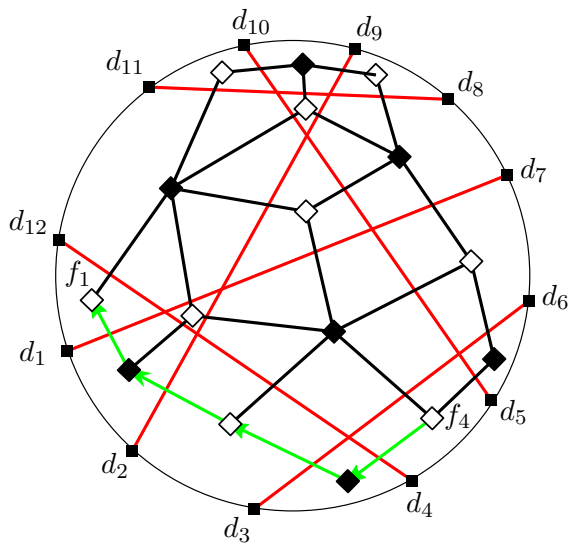
  \caption{A path in the diamond graph going from $f_4$ to $f_1$}\label{path}
\end{figure}

\begin{prop}\label{main_lemma}
  Let $R=(\tau,\bm{\alpha})$ be a non-alternating region and $j \in [2n]$ be a
  $\tau$-descent. Then for $R'=R\cdot t_j$ we have for all $t\in\R$,
  \[ \gamma_R(t) = \gamma_{R'}(t)\cdot g_j^{\bm{\alpha}}, \]
  where $g_j^{\bm{\alpha}}$ is the matrix defined in Definition~\ref{g_p}.
\end{prop}

\begin{proof}
  In this proof, we consider the tilded version of $\tau, \alpha$ and $J_p$ even
  if the tilde notation is omitted. It allows us to treat the case $j=2n$ with the
  same notation as the other cases $j< 2n$. For the region $R'$, write
  $\gamma_{R'}(t)=(\gamma'_1(t),\ldots,\gamma'_{2n}(t))$ and, for $m\in [2n]$,
  denote by $J'_{m}$ the set obtained from $\tau'$. By comparing $J_m$ and $J'_m$
  for all $m$ different from $j,j+1$ we see that $\gamma_m(t)=\gamma'_m(t)$. So it
  remains to consider the case $m=j$ and the case $m=j+1$.

  Note that
  \[
    J_j\setminus\{j+1\}=J_{j+1}\setminus\{\tau(j)\} \text{ and }
    J'_j\setminus\{j+1\}=J'_{j+1}\setminus\{\tau'(j)\}.
  \]

  Suppose first that $j$ is even.

  We have
\[ (\gamma_{R'}(t)\cdot g_j^{\bm{\alpha}})_j =
  \gamma'_j(t)\frac{1}{\mathsf{cn}(\alpha_{j+1}-\alpha_j)}+\gamma'_{j+1}(t)\frac{\mathsf{sn}(\alpha_{j+1}-\alpha_j)}{\mathsf{cn}(\alpha_{j+1}-\alpha_j)}
\]
  and 
  \[ (\gamma_{R'}(t)\cdot g_j^{\bm{\alpha}})_{j+1} = \gamma'_j(t)\frac{\mathsf{sn}(\alpha_{j+1}-\alpha_j)}{\mathsf{cn}(\alpha_{j+1}-\alpha_j)}+\gamma'_{j+1}(t)\frac{1}{\mathsf{cn}(\alpha_{j+1}-\alpha_j)}. \]

  Since $j$ is even, we have $\Pi_j(t)=\Pi_{j+1}(t)$ so, by (\ref{gamma}),
  $\gamma'_j(t)=\Phi(t)\mathsf{sn}(t-\alpha'_{j+1})$ and
  $\gamma'_{j+1}(t)=\Phi(t)\mathsf{sn}(t-\alpha'_{\tau'(j)})$ where 
  \[ \Phi(t):=\prod_{l\in J'_j\setminus
    \{j+1\}}\Pi_j(t)\mathsf{sn}(t-\alpha'_l)
  \] 
  is the common factor.
  Recall that, by Condition~\eqref{cond_1} on the angles and
  Definition~\ref{region_remove} of~$R'$, the identities
  $\alpha'_{\tau'(j)}=\alpha_{\tau(j+1)}=\alpha_{j+1}+\frac{\pi}{2}$ and
  $\alpha'_{j+1}=\alpha_j$ are satisfied. Therefore
  \[
    \gamma'_{j+1}(t)=-\Phi(t)\mathsf{cd}(t-\alpha_{j+1})
  \]
  where
  $\mathsf{cd}(u):=\frac{\mathsf{cn}(u)}{\mathsf{sn}(u)}$ for all $u\in\R$. We
  also have 
  \[
    \gamma'_j(t)=\Phi(t)\mathsf{sn}(t-\alpha_j).
  \]
  We deduce that
  \begin{multline*}
    (\gamma_{R'}(t)\cdot g_j^{\bm{\alpha}})_j  = 
    \frac{\Phi(t)}{\mathsf{cn}(\alpha_{j+1}-\alpha_j)}
    \\ \times\Bigl(\mathsf{sn}\bigl[(t-\alpha_{j+1})+(\alpha_{j+1}-\alpha_j)\bigr]
    - \mathsf{cd}(t-\alpha_{j+1})\mathsf{sn}(\alpha_{j+1}-\alpha_j)\Bigr).
  \end{multline*}
  Now using Relation~\eqref{sn_formula} 
  \begin{equation*}
    \dn(u)\sn(u+v)= \cn(u)\sn(v)+\sn(u)\cn(v)\dn(u+v),
  \end{equation*}
  we get
  \begin{multline*}
    (\gamma_{R'}(t)\cdot g_j^{\bm{\alpha}})_j =
    \frac{\Phi(t)}{\mathsf{cn}(\alpha_{j+1}-\alpha_j)} 
    \Bigl(\frac{\mathsf{cn}(t-\alpha_{j+1})\mathsf{sn}(\alpha_{j+1}-\alpha_j)}{\mathsf{dn}(t-\alpha_{j+1})} \\
      +\frac{\mathsf{sn}(t-\alpha_{j+1})\mathsf{cn}(\alpha_{j+1}-\alpha_j)\mathsf{dn}(t-\alpha_j)}{\mathsf{dn}(t-\alpha_{j+1})} \Bigr.
    \Bigl.-\mathsf{cd}(t-\alpha_{j+1})\mathsf{sn}(\alpha_{j+1}-\alpha_j)\Bigr),
  \end{multline*}
  thus
  \[ (\gamma_{R'}(t)\cdot g_j^{\bm{\alpha}})_j = \Phi(t)\mathsf{sn}(t-\alpha_{j+1})\frac{\mathsf{dn}(t-\alpha_j)}{\mathsf{dn}(t-\alpha_{j+1})}. \]
  Injecting the expression of $\Phi(t)$ and reorganizing the terms we have
  \begin{multline*}
    (\gamma_{R'}(t)\cdot g_j^{\bm{\alpha}})_j = \Bigl(\prod_{l\in J'_j\setminus\{j+1\}}\mathsf{sn}(t-\alpha'_{l})\Bigr)\mathsf{sn}(t-\alpha_{j+1}) \\
    \times\Bigl(\prod_{l=1}^{\lfloor \frac{j}{2}\rfloor-1}\frac{\mathsf{dn}(t-\alpha'_{2l})}{\mathsf{dn}(t-\alpha'_{\tau'(2l-1)})} \Bigr)\frac{\mathsf{dn}(t-\alpha'_j)}{\mathsf{dn}(t-\alpha'_{\tau'(j-1)})}\frac{\mathsf{dn}(t-\alpha_j)}{\mathsf{dn}(t-\alpha_{j+1})}.
  \end{multline*}
  Since for all $l\in J'_j\setminus \{j+1\}$ we have $\alpha'_l=\alpha_l$ and
  since $J'_j=J_j$, it gives
  \[ \Bigl(\prod_{l\in J'_j\setminus\{j+1\}}\mathsf{sn}(t-\alpha'_{l})\Bigr)\mathsf{sn}(t-\alpha_{j+1})
  = \prod_{l\in J_j}\mathsf{sn}(t-\alpha_l). \]
  Similarly, for all $1\leq l \leq \lfloor \frac{j}{2} \rfloor -1$, we have
  $\alpha'_{2l}=\alpha_{2l}$ and $\alpha'_{\tau'(2l-1)}=\alpha_{\tau(2l-1)}$. We
  also have $\alpha'_j=\alpha_{j+1}$ and
  $\alpha'_{\tau'(j-1)}=\alpha_{\tau(j-1)}$ so
  \begin{equation} \left(\prod_{l=1}^{\lfloor \frac{j}{2}\rfloor-1}\frac{\mathsf{dn}(t-\alpha'_{2l})}{\mathsf{dn}(t-\alpha'_{\tau'(2l-1)})} \right)\frac{\mathsf{dn}(t-\alpha'_j)}{\mathsf{dn}(t-\alpha'_{\tau'(j-1)})}\frac{\mathsf{dn}(t-\alpha_j)}{\mathsf{dn}(t-\alpha_{j+1})} = \prod_{l=1}^{\lfloor \frac{j}{2} \rfloor} \frac{\mathsf{dn}(t-\alpha_{2l})}{\mathsf{dn}(t-\alpha_{\tau(2l-1)})}.
    \label{bigequa} 
  \end{equation}
  Finally we get
  \[
    (\gamma_{R'}(t)\cdot g_j^{\bm{\alpha}})_j = \prod_{l\in
    J_j}\mathsf{sn}(t-\alpha_l)\prod_{l=1}^{\lfloor \frac{j}{2} \rfloor}
    \frac{\mathsf{dn}(t-\alpha_{2l})}{\mathsf{dn}(t-\alpha_{\tau(2l-1)})}=
    \gamma_j(t).
  \]

Now we turn to the coefficient $j+1$. Factorizing by $\Phi(t)$ we have
\begin{multline*} (\gamma_{R'}(t)\cdot g_j^{\bm{\alpha}})_{j+1} = \frac{\Phi(t)}{\mathsf{cn}(\alpha_{j+1}-\alpha_j)}\Bigl(\mathsf{sn}(t-\alpha_j)\mathsf{sn}(\alpha_{j+1}-\alpha_j)\\-\mathsf{cd}\bigl[(t-\alpha_j)-(\alpha_{j+1}-\alpha_j)\bigr]\Bigr).
\end{multline*}
Then we use the following relation~\eqref{cn_formula}
\[
  \cn(u+v)=\cn(u)\cn(v) - \sn(u)\sn(v)\dn(u+v)
\]
 to get
\begin{multline*}
(\gamma_{R'}(t)\cdot g_j^{\bm{\alpha}})_{j+1} = \frac{\Phi(t)}{\mathsf{cn}(\alpha_{j+1}-\alpha_j)}\Bigl(\mathsf{sn}(t-\alpha_j)\mathsf{sn}(\alpha_{j+1}-\alpha_j)-\frac{\mathsf{cn}(t-\alpha_j)\mathsf{cn}(\alpha_{j+1}-\alpha_j)}{\mathsf{dn}(t-\alpha_{j+1})}\Bigr.  \\
\Bigl.-\frac{\mathsf{sn}(t-\alpha_j)\mathsf{sn}(\alpha_{j+1}-\alpha_j)\mathsf{dn}(t-\alpha_{j+1})}{\mathsf{dn}(t-\alpha_{j+1})}\Bigr)
\end{multline*}
Thus
\[ (\gamma_{R'}(t)\cdot g_j^{\bm{\alpha}})_{j+1} = \Phi(t)(-\mathsf{cd}(t-\alpha_j))\frac{\mathsf{dn}(t-\alpha_j)}{\mathsf{dn}(t-\alpha_{j+1})} = \Phi(t)\mathsf{sn}(t-\alpha_{\tau(j)})\frac{\mathsf{dn}(t-\alpha_j)}{\mathsf{dn}(t-\alpha_{j+1})}. \]
Reinjecting the expression of $\Phi(t)$ and reorganizing the terms it gives
\begin{multline*}
 (\gamma_{R'}(t)\cdot g_j^{\bm{\alpha}})_{j+1} =
 \left(\prod_{l\in J'_j\setminus\{j+1\}}\mathsf{sn}(t-\alpha'_{l})\right)
 \mathsf{sn}(t-\alpha_{\tau(j)}) \\
 \times
 \left(\prod_{l=1}^{\lfloor \frac{j}{2}\rfloor-1}\frac{\mathsf{dn}(t-\alpha'_{2l})}{\mathsf{dn}(t-\alpha'_{\tau'(2l-1)})} \right)\frac{\mathsf{dn}(t-\alpha'_j)}{\mathsf{dn}(t-\alpha'_{\tau'(j-1)})}\frac{\mathsf{dn}(t-\alpha_j)}{\mathsf{dn}(t-\alpha_{j+1})}.
\end{multline*}
Since $J_{p+1}=J'_p\setminus \{ p+1 \} \cup \{\tau(p)\}$ we have
\[ \left(\prod_{j\in J'_p\setminus\{p+1\}}\mathsf{sn}(t-\alpha'_{j})\right)\mathsf{sn}(t-\alpha_{\tau(p)})
= \prod_{j\in J_{p+1}}\mathsf{sn}(t-\alpha_j).\]
Recalling~\eqref{bigequa} and since $\lfloor \frac{j}{2} \rfloor= \lfloor
\frac{j+1}{2} \rfloor$ for $j$ even we finally deduce that
\[ (\gamma_{R'}(t)\cdot g_j^{\bm{\alpha}})_{j+1} = \prod_{l\in J_{j+1}}\mathsf{sn}(t-\alpha_l)\prod_{l=1}^{\lfloor \frac{j+1}{2} \rfloor} \frac{\mathsf{dn}(t-\alpha_{2l})}{\mathsf{dn}(t-\alpha_{\tau(2l-1)})}= \gamma_{j+1}(t). \]

Suppose now that $j$ is odd. The idea is the same but the common factor between
$\gamma'_j(t)$ and $\gamma'_{j+1}(t)$ is different. Indeed we have,
\[
  \gamma'_j(t)=\Phi(t)\mathsf{sn}(t-\alpha'_{j+1})=
\Phi(t)\mathsf{sn}^*(t-\alpha'_{j+1})\mathsf{dn}^*(t-\alpha'_{\tau'(j+1)}).
\]
And using consecutively Table~\ref{table_formula} and the dual relation
\eqref{dual_dn} we have
\begin{align*}
\gamma'_{j+1}(t)& =\Phi(t)\mathsf{sn}(t-\alpha'_{\tau'(j)})\frac{\mathsf{dn}(t-\alpha'_{j+1})}{\mathsf{dn}(t-\alpha'_{\tau'(j)})} \\
	& = \Phi(t)\mathsf{sn}^*(t-\alpha'_{\tau'(j)})\frac{1}{k'}\mathsf{dn}(t-\alpha'_{j+1}) \\
	& = \Phi(t)\mathsf{sn}^*(t-\alpha'_{\tau'(j)})\mathsf{dn}^*(t-\alpha'_{\tau'(j+1)}) .
\end{align*}
Thus $\gamma'_j(t)=\tilde{\Phi}(t)\mathsf{sn}^*(t-\alpha'_{j+1})$ and
$\gamma'_{j+1}(t)=\tilde{\Phi}(t)\mathsf{sn}^*(t-\alpha'_{\tau'(j)})$ with
$\tilde{\Phi}(t)=\Phi(t)\mathsf{dn}^*(t-\alpha'_{\tau'(j+1)})$ is the common
factor between $\gamma'_j(t)$ and $\gamma'_{j+1}(t)$. A similar computation as
above with $\mathsf{sn}^*$ instead of $\mathsf{sn}$ therefore gives
\[ (\gamma_{R'}(t)\cdot g_j^{\bm{\alpha}})_j = \tilde{\Phi}(t)\mathsf{sn}^*(t-\alpha_{j+1})\frac{\mathsf{dn}^*(t-\alpha_j)}{\mathsf{dn}^*(t-\alpha_{j+1})}. \]
Rewriting these without $*$, using the duality relations we have
\[ (\gamma_{R'}(t)\cdot g_j^{\bm{\alpha}})_j = \Phi(t) \mathsf{sn}(t-\alpha_{j+1})\frac{\mathsf{dn}(t-\alpha_{\tau(j+1)})}{\mathsf{dn}(t-\alpha'_{\tau'(j+1)})}\frac{\mathsf{dn}(t-\alpha_{j+1})}{\mathsf{dn}(t-\alpha_j)}. \]
Since $\alpha'_{\tau'(j+1)}=\alpha_{\tau(j)}$ we deduce that
\[ \frac{\mathsf{dn}(t-\alpha_{\tau(j+1)})}{\mathsf{dn}(t-\alpha'_{\tau'(j+1)})}\frac{\mathsf{dn}(t-\alpha_{j+1})}{\mathsf{dn}(t-\alpha_j)} = 1.\]
Finally we get
\[ (\gamma_{R'}(t)\cdot g_j^{\bm{\alpha}})_j= \prod_{l\in J_j}\mathsf{sn}(t-\alpha_l)\prod_{l=1}^{\lfloor \frac{j}{2} \rfloor} \frac{\mathsf{dn}(t-\alpha_{2l})}{\mathsf{dn}(t-\alpha_{\tau(2l-1)})} = \gamma_j(t). \]

Now we turn to the coefficient $j+1$. By the above computations we have
\[ (\gamma_{R'}(t)\cdot g_j^{\bm{\alpha}})_{j+1} = \tilde{\Phi}(t)\mathsf{sn}^*(t-\alpha_{\tau(j)})\frac{\mathsf{dn}^*(t-\alpha_j)}{\mathsf{dn}^*(t-\alpha_{j+1})}. \]
Rewriting these using the duality relations we have
\[ (\gamma_{R'}(t)\cdot g_j^{\bm{\alpha}})_{j+1} = \Phi(t)\mathsf{sn}(t-\alpha_{\tau(j)})\frac{\mathsf{dn}(t-\alpha_{j})}{\mathsf{dn}(t-\alpha'_{\tau'(j+1)})}\frac{\mathsf{dn}(t-\alpha_{j+1})}{\mathsf{dn}(t-\alpha_j)}.\]
Since $\alpha'_{\tau'(j+1)}=\alpha_{\tau(j)}$ we deduce that
\[ (\gamma_{R'}(t)\cdot g_j^{\bm{\alpha}})_{j+1} = \prod_{l\in J_j}\mathsf{sn}(t-\alpha_l)\prod_{l=1}^{\lfloor \frac{j+1}{2} \rfloor} \frac{\mathsf{dn}(t-\alpha_{2l})}{\mathsf{dn}(t-\alpha_{\tau(2l-1)})} = \gamma_{j+1}(t) \qedhere. \]

\end{proof}

Now we can turn to the proof of the main theorem:

\begin{teo*}
Let $R=(\tau,\bm{\alpha})$ be a region and $j\in[2n]$ be a $\tau$-descent. Then for $R':=R\cdot t_j$, we have
\[ \phi(M_R)= \phi(M_{R'})\cdot g_j^{\bm{\alpha}} \]
\end{teo*}

\begin{proof}[Proof of Theorem~\ref{main_teo}]
As in the paper~\cite{galashin2021critical}, we proceed by induction. Our proof
only differs in the induction step since we are careful with the fact that a
crossing does not imply the existence of a $\tau$-descent, see
Remark~\ref{no_tau_descent}. We proceed by induction on the number $\ell$ of
crossings.

For $\ell=0$, let $R$ be a region without crossing. Since there is no crossing,
there is always an index $p$ such that $\tau(p)=p+1$. Let us show that
$\Span(\gamma_R)$ contains the vector $e_p+\varepsilon_{p,\tau(p)}e_{\tau(p)}$
where $\varepsilon_{p,\tau(p)}=(-1)^{\frac{\tau(p)-p-1}{2}}$ and
$e_1,...,e_{2n}$ is the canonical basis of $\R^{2n}$.

By the non-alternating condition we have that
$\gamma_p(\alpha_p),\gamma_{\tau(p)}(\alpha_p)\neq 0$ while
$\gamma_m(\alpha_p)=0$ for all $m\neq p,\tau(p)$. Note that $J_p=J_{\tau(p)}$
and $\Pi_p=\Pi_{\tau(p)}$ so
$\gamma_p(t)=\varepsilon_{p,\tau(p)}\gamma_{\tau(p)}(t)$. Thus
$\gamma(\alpha_p)$ is proportional to $e_p+\varepsilon_{p,\tau(p)}e_{\tau(p)}$
which shows the result.

Now we can modify the curve $\gamma_R$ as follows. Define
$\gamma_R^{[p,\tau(p)]}$ from $\gamma_R$ by setting the $p$-th and the
$\tau(p)$-th coordinates to $0$ and dividing the other coordinates by
$\sin(t-\alpha_p)$ (which is a common factor of all $\gamma_m(t)$ where $m\neq
p,\tau(p)$). Let us show that
\[
  \Span(\gamma_R)=\Span(\gamma_R^{[p,\tau(p)]})
\oplus \Span(e_p+\varepsilon_{p,\tau(p)}e_{\tau(p)}).
\]
The first inclusion is clear since we have
\[ \gamma_R(t)=\mathsf{sn}(t-\alpha_p)\gamma_R^{[p,\tau(p)]}(t)+\prod_{j\in J_p}\mathsf{sn}(t-\alpha_j)\Pi_p(e_p+\varepsilon_{p,\tau(p)}). \]
For the other inclusion it suffices to note that
$\gamma_R(\alpha_p)=c(e_p+\varepsilon_{p,\tau(p)})$ where $c$ is a constant and
that
\[ \gamma_R^{[p,\tau(p)]}(t) = \frac{1}{\mathsf{sn}(t-\alpha_p)}\gamma_R(t)-\prod_{j\in J_p}\mathsf{sn}(t-\alpha_j)\Pi_p\frac{1}{c}\gamma_R(\alpha_p). \]

Now notice that $\Span(\gamma_R^{[p,\tau(p)]})$ with two of its coordinates set
to $0$ is the same as $\Span(\gamma_{\tilde{R}})$  where $\tilde{R}$ is the
region obtained from $R$ by removing the arc ${p,\tau(p)}$. Indeed the product
$\Pi_m$ associated to $m$ for all $m$ not equal to $p,\tau(p)$ is the same in
$R$ and in $\tilde{R}$. This can be seen for instance using the path notation of
the product $\Pi_m$ explained in Remark~\ref{remark_formula} and using the fact
that we can find a path that does not cross the traintrack linking $d_p$ to
$d_{\tau(p)}$.
Therefore we can repeat the process for the region $\tilde{R}$ and eventually
decompose $\Span(\gamma_R)$ as a direct sum of
$\Span(e_p+\varepsilon_{p,\tau(p)}e_{\tau(p)})$ over all pairs $p<\tau(p)$. By
definition of $\phi(M_R)$ as the span of the matrix $\tilde{M}$, this direct
sum is equal to $\phi(M_R)$. It also proves that for $\ell=0$, $\Span(\gamma_R)$
is of dimension $n$. Since for all $j\in[2n]$ and for all $\bm{\alpha}\in
\R^{2n}$, the matrix $g_j^{\alpha}$ is invertible, we can conclude that
$\Span(\gamma_R)$ is of dimension $n$ for all $\ell\geq 0$.

Now for the induction step, suppose that the result is true for $\ell$
crossings. Let $R$ be a region with $\ell+1$ crossings. Let $j<k$  be two
indices of $[2n]$ such that the traintrack linking $d_j$ and $d_{\tau(j)}$
crosses the traintrack linking $d_k$ and $d_{\tau(k)}$. Assume the difference
$k-j$ is minimal in the sense that no other $k,j$ with a smaller difference
satisfy the above condition.
\begin{itemize}
\item If $j$ is a $\tau$-descent (\emph{i.e.} $k=j+1$). Then let $R'=R\cdot
  t_j$. $R'$ has $\ell$ crossings so by the induction hypothesis we have
  $\Span(\gamma_{R'})=\phi(M_{R'})$. Since $\phi(M_R)=\phi(M_{R'})\cdot
  g_j^{\alpha}$ by Theorem~\ref{theo_phi} and
  $\Span(\gamma_{R})=\Span(\gamma_{R'})\cdot g_j^{\alpha}$ by Proposition
 ~\ref{main_lemma} we deduce that $\phi(M_R)=\Span(\gamma_R)$.
\item If $j$ is not a $\tau$-descent, we transform the region $R$ into a region
  $\hat{R}$ such that $j$ is a $\tau$-descent in $\hat{R}$ as follows. First
  note that, by minimality of the difference $k-j$, it means that there is no
  crossing in the area between $j+1$ and $k-1$, see Figure~\ref{algo_descent}. 

\begin{figure}[ht]
\centering
\def\svgwidth{10cm}
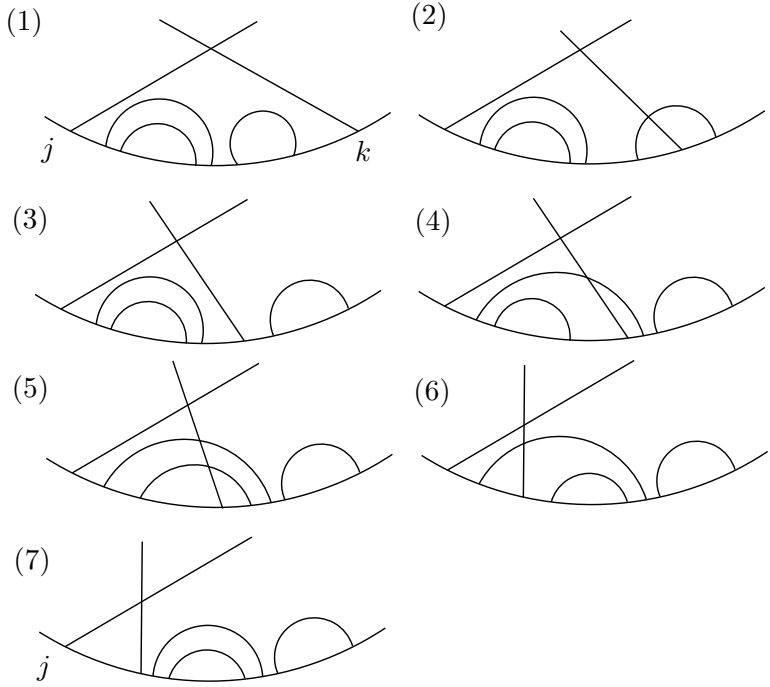
\caption{Procedure to make $j$ into a $\tau$-descent (from left to right and top to bottom)}
\label{algo_descent}
\end{figure}

Define, by induction, a sequence of regions $(R_i)_{0\leq i\leq
k-(j+1)}=(\tau_i,\bm{\alpha}_i)$ by $R_0:=R$ and for $i\geq 1$:
\begin{itemize}
  \item If $k-i$ is a $\tau_{i-1}$-descent, then $R_{i}=R_{i-1}\cdot t_{k-i}$.
  \item Otherwise, set $R_{i}$ to be the region obtained by crossing $k-i$ and $k-(i-1)$.
\end{itemize}
The region $\hat{R}:=R_{k-(j+1)}$ is such that $j$ is a $\hat{\tau}$-descent.
And we have $\phi(M_R)=\phi(M_{\tilde{R}})\cdot g$ and
$\Span(\gamma_R)=\Span(\gamma_{\hat{R}})\cdot g$ where $g$ is an invertible
matrix. By the first step applied to $\hat{R}$ we know that
$\Span(\gamma_{\hat{R}})=\phi(M_{\hat{R}})$ so we deduce the result for $R$.
\qedhere
\end{itemize}
\end{proof}

We end with a proof of Proposition~\ref{prop_basis}.

\begin{proof}[Proof of Proposition~\ref{prop_basis}]
Let $A$ be the matrix with rows
$\gamma_R(\alpha_{j_1}),\dotsc,\gamma_R(\alpha_{j_n})$. Note that, since all
angles are distinct, we have $\gamma_{j_l}(\alpha_{j_l})\neq 0$ for all $l\in
[n]$. Therefore, the principal minor of $A$ composed by the $n$ first columns of
$A$ is an upper-triangular matrix with non-zero diagonal coefficients. It is
thus invertible which implies that $A$ is of rank $n$.
\end{proof}


\subsection{Alternating regions}
\label{alternating_regions}

In this section we extend our result to alternating regions. 

\begin{defi}
Let $R=(\tau,\bm{\alpha})$ and $v_j=e^{i2\alpha_j}$ for all $j\in [2n]$. An
\emph{alternating region} is a region such that there exist integers
$1\leq i<j<k<l\leq 2n$ such that $v_i=-v_j=v_k=-v_l$.
\end{defi}

We adapt the ideas of~\cite{galashin2022formula} to our case.  Let
$R=(\tau,\bm{\alpha})$ be a region. For $p\in [2n]$, define $m_p:=\#\{j\in J_p :
v_j=v_p\}$. For $p\in [2n]$, let $\supp_{\tau}(p)$ be the set of indices that we
encounter by going from $p$ to $\tau(p)$ in the trigonometric direction, that is
\[ \supp_{\tau}(p):=\left\{
    \begin{array}{ll}
        \{p,p+1,\dotsc,\tau(p)\} & \mbox{if } p<\tau(p) \\
        \{p,p+1,\dotsc,n,1,\dotsc,\tau(p)\} & \mbox{if } p > \tau(p).
    \end{array}
\right.
\]
For a vector $x\in \C^{2n}$ and a set $S\subset [2n]$, define the vector
$x_{|S}$ to be the vector obtained from $x$ by

\[
(x_{|S})_p := \left\{
    \begin{array}{ll}
        x_p & \mbox{if } p \in S \\
        0 & \mbox{if } p \notin S.
    \end{array}
\right.
\]
Finally, for $p\in [2n]$, define
\[ u_R^{(p)}:=\gamma_R^{(m_p)}(\alpha_p)_{|\supp_{\tau}(p)} \]
where $\gamma_R^{(m_p)}$ is the $m_p$-th derivative of $\gamma_R$.
\begin{figure}[ht!]
\centering
\def\svgwidth{6cm}
\begingroup%
  \makeatletter%
  \providecommand\color[2][]{%
    \errmessage{(Inkscape) Color is used for the text in Inkscape, but the package 'color.sty' is not loaded}%
    \renewcommand\color[2][]{}%
  }%
  \providecommand\transparent[1]{%
    \errmessage{(Inkscape) Transparency is used (non-zero) for the text in Inkscape, but the package 'transparent.sty' is not loaded}%
    \renewcommand\transparent[1]{}%
  }%
  \providecommand\rotatebox[2]{#2}%
  \newcommand*\fsize{\dimexpr\f@size pt\relax}%
  \newcommand*\lineheight[1]{\fontsize{\fsize}{#1\fsize}\selectfont}%
  \ifx\svgwidth\undefined%
    \setlength{\unitlength}{150.20162591bp}%
    \ifx\svgscale\undefined%
      \relax%
    \else%
      \setlength{\unitlength}{\unitlength * \real{\svgscale}}%
    \fi%
  \else%
    \setlength{\unitlength}{\svgwidth}%
  \fi%
  \global\let\svgwidth\undefined%
  \global\let\svgscale\undefined%
  \makeatother%
  \begin{picture}(1,1.00577967)%
    \lineheight{1}%
    \setlength\tabcolsep{0pt}%
    \put(0,0){\includegraphics[width=\unitlength,page=1]{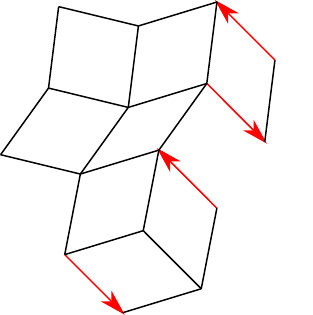}}%
    \put(0.18975972,0.06411899){\color[rgb]{0,0,0}\makebox(0,0)[lt]{\lineheight{1.25}\smash{\begin{tabular}[t]{l}$v_i$\end{tabular}}}}%
    \put(0.61529578,0.43008757){\color[rgb]{0,0,0}\makebox(0,0)[lt]{\lineheight{1.25}\smash{\begin{tabular}[t]{l}$v_j$\end{tabular}}}}%
    \put(0.66796311,0.60629375){\color[rgb]{0,0,0}\makebox(0,0)[lt]{\lineheight{1.25}\smash{\begin{tabular}[t]{l}$v_k$\end{tabular}}}}%
    \put(0.78617181,0.92733688){\color[rgb]{0,0,0}\makebox(0,0)[lt]{\lineheight{1.25}\smash{\begin{tabular}[t]{l}$v_l$\end{tabular}}}}%
  \end{picture}%
\endgroup%

\caption{Example of an alternating region}
\label{alternating_region}
\end{figure}
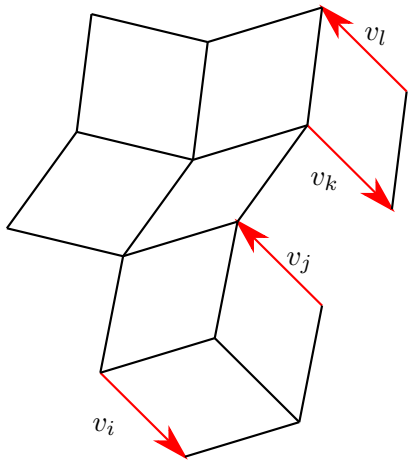

\begin{prop}
Let $R=(\tau,\bm{\alpha})$ be a region. The $n$ vectors $\{u_R^{(j)}:j\in J_1\cup\{1\}\}$ form a basis of $\phi(M_R)$.
\end{prop}

\begin{proof}
Let $(R_N=(\tau,\bm{\alpha}^N)_{N\geq 1}$ be a sequence of regions such that for
all $N\geq 1$, the angles $(\alpha^N_j)_{j\in [2n]}$ are all distinct (so that
in particular $R_N$ is non-alternating for all $N$) and $\alpha^N_j$ tends to
$\alpha_j$ when $N$ goes to infinity. The function $R\rightarrow M_R$ is
continuous in $\bm{\alpha}$ so $M_{R_N}\rightarrow M_R$ as $N\rightarrow\infty$.
Furthermore the doubling map $\phi$ is continuous so $\phi(M_R)$ consists of
vectors of $\R^{2n}$ that are limits of sequences of vectors $x^N$ of
$\phi(M_{R_N})=\Span(\gamma_{R_N})$. Therefore we should prove that the vectors
$\{u_R^{(j)}:j\in J_1\cup\{1\}\}$ are obtained as such limits and that they are
linearly independent. 

We first show that these vectors are linearly independent. Denote
$J_1\cup\{1\}:=\{j_1<j_2<\dotsc<j_n\}$. Let $U$ be the matrix with lines
$u_R^{j_1},u_R^{j_2},\dotsc,u_R^{j_n}$. Consider the $n\times n$ principal minor
indexed by the columns $J_1\cup\{1\}$. We claim that this minor is an
upper-triangular matrix with non zero diagonal coefficients. Indeed, since we
differentiate exactly the right number of times, we have
$(u_R^{j_l})_{j_l}=\gamma_{j_l}^{(m_{j_l})}(\alpha_{j_l})\neq 0$. For $1\leq
l<l'\leq n$, we have $l\notin \supp_{\tau}(l')$ thus $(u_R^{j_{l'}})_{j_l}=0$.

Now we show that for each $j\in[2n]$, $u_R^{j}$ is a limit of a sequence
$u^{(j,N)}\in\Span(\gamma_{R_N})$. In other words, we want to approximate the
derivatives of $\gamma_R$ by linear combinations of $\gamma_{R_N}$. To do that,
we introduce the \emph{discrete derivative operator}. For a smooth function $f$
and for $a\in\R$ define
\[\bar{\partial}_af(x) := \frac{f(x)-f(a)}{x-a} \]
For $(a_1,a_2,\dotsc,a_m)$ distinct real numbers we set
\[ \bar{\partial}_{(a_1,\dotsc,a_m)}f(x):=\bar{\partial}_{a_1}\circ\bar{\partial}_{a_2}\circ\dotsc\bar{\partial}_{a_m}. \]
We have the following property that can be proved by induction, using Taylor expansion:
\begin{equation} \lim_{a_1,\dotsc,a_m,a\rightarrow a_0}[(\bar{\partial}_{(a_1,\dotsc,a_m)}f(x))_{|x=a}]=\frac{1}{m!}f^{(m)}(a_0). 
\label{discrete_limit}
\end{equation}
Let $j\in [2n]$. Define
\[ J'_j=\{l\in J_j\cap \supp_{\tau}(j):\alpha_l=\alpha_j\} \text{ and } J''_j=\{l\in J_j\setminus\supp_{\tau}(j):\alpha_l=\alpha_j\}.\]
Write $J'_j:=\{l'_1,\dotsc,l'_{m'}\}$ and $J''_j:=\{l''_1,\dotsc,l''_{m''}\}$. We have $m_j=m'+m''$. Now define
\[ \gamma^{\flat}_{R_N}(t)=\frac{1}{\prod_{s=1}^{m''}\mathsf{sn}(t-\alpha^N_{l''_s})}\gamma_{R_N}(t).\]
Finally, let
\[ w^{(j,N)}:=\bar{\partial}_{(\alpha^N_{l'_1},\dotsc,\alpha^N_{l'_{m'}})}\gamma^{\flat}_{R_N}(\alpha^N_j).\]
These vectors are linear combinations of vectors of $\Span(\gamma_{R_N})$ and we
now show that they tend to $u^{(j)}_R$, up to a multiplicative constant, when
$N$ goes to infinity. First let us check that the coordinates that are not in
$\supp_{\tau}(j)$ are equal to $0$. Let $p\notin \supp_{\tau}(j)$. Then $j\in
J_p$ thus $\gamma_p(\alpha_j)=0$. Furthermore, for all $s\in [m']$, we have
$l'_s\in J'_j$  thus $\supp_{\tau}(l'_s)\subset\supp_{\tau}(j)$. Therefore
$p\notin \supp_{\tau}(l'_s)$ and $\gamma_p(\alpha_{l'_s})=0$. So for $p\notin
\supp_{\tau}(j)$, $w^{(j,N)}_p\rightarrow 0$ when $N$ goes to infinity. It
remains to show that the other coordinates converge to the expected result. We
know, by the property~\eqref{discrete_limit}, that $m'!w^{(j,N)}$ converges to
$(\gamma^{\flat}_R)^{(m')}(\alpha_j)$ when $N\rightarrow \infty$. Denote by
$g_N(t)$ the function $\prod_{s=1}^{m''}\mathsf{sn}(t-\alpha^N_{l''_s})$ and by
$g(t)$ the limit of $g_N(t)$ as $N\rightarrow \infty$. By Leibniz formula, we
have
\[ (\gamma_{R_N})^{(m_j)}(\alpha_j) = \sum_{s=0}^{m_j}\binom{m_j}{s}(\gamma^{\flat}_{R_N})^{(s)}(\alpha_j)g_N^{(m_j-s)}(\alpha_j). \]
Note that for $m_j-s<m''$ \emph{i.e.} $s>m'$, we have
$g_N^{(m_j-s)}(\alpha_j)\rightarrow 0$ as $N\rightarrow \infty$. And for $s<m'$
we have $(\gamma^{\flat}_{R_N})^{(s)}(\alpha_j)\rightarrow 0$ as $N\rightarrow
\infty$. Therefore
\[ \gamma_R^{(m_j)}(\alpha_j)=\binom{m_j}{m'}g^{(m'')}(\alpha_j)(\gamma^{\flat}_R)^{m'}(\alpha_j). \]
We have therefore proved that $u^{(j)}_R$ is a limit of vectors of $\Span(\gamma_{R_N})$ which finishes the proof.

\end{proof}

\section{Kramers-Wannier Duality}
\label{section_duality}

Let $\mathscr{T}$ be a traintrack arrangement. We have considered an Ising model
on $G^{\bullet}_{\mathscr{T}}$ with weights $(j_e)$ defined by
\eqref{Z_invariant_weights}. There is a dual Ising model defined on the dual
graph $G^{\circ}_{\mathscr{T}}$ with weights $(j^*_e)$ that satisfy the
Kramers-Wannier duality relation:
\[ \forall e\in E \quad \sinh(2j_e)\sinh(2j^*_{e^*})=1. \]
It was proved in~\cite[Section~4]{boutillier2019z} that the dual weights $j^*_{e^*}$ have the same expression than $j_e$ with the parameter $k$ replaced by the dual parameter $k^*$.
In our setting, an Ising model on the dual graph $G^{\circ}_{\mathbb{A}}$ can be
obtained by switching the black and white colors of the vertices of the diamond
graph and relabeling the vectors $(v_1,\dotsc,v_{2n})$ into
$(v_2,\dotsc,v_{2n},v_1)$. Both ways give in fact the same model on the dual
graph. More precisely, let $S:\R^{2n}\rightarrow \R^{2n}$ be the \emph{cyclic
shift operator} defined by $S(x_1,x_2\dotsc,x_{2n})=((-1)^{n-1}x_{2n}, x_1,x_2,
\dotsc, x_{2n-1})$. The following theorem is proved
in~\cite[Theorem~3.4]{galashin2020ising}.

\begin{teo}[\cite{galashin2020ising}]
Let $R=(\tau,\bm{\alpha})$ be a region and let $\mathscr{T}$ be a traintrack
arrangement. Let $(G^{\bullet}_{\mathscr{T}},j)$ be an Ising model and
$(G^{\circ}_{\mathscr{T}},j^*)$ be the Kramers-Wannier dual Ising model. Let $M$
and $M^*$ be the respective correlation matrices. Then
\[ \phi(M)\cdot S = \phi(M^*). \]
\end{teo}

For all $p\in [2n]$, define
\[ \gamma^*_p(t)=\prod_{j=1}^{2\lfloor\frac{p}{2}\rfloor} \frac{\mathsf{dn^*}(t-\tilde{\alpha}_j)}{\sqrt{(k^*)'}}\prod_{j\in \tilde{J}_p}\mathsf{sn^*}(t-\tilde{\alpha}_j). \]
As noticed in Remark~\ref{remark_formula}, computations work in the same way
with the dual parameter~$k^*$, in particular the proof of Theorem~\ref{main_teo}
is well-adapted and we have $\phi(M^*)=\Span(\gamma^*_R)$. Note that, thanks to
the tilded notation, we get the equality
\[ \begin{pmatrix}
\gamma_1(t) & \gamma_2(t) & \dotsc & \gamma_{2n}(t)
\end{pmatrix}
\cdot S = \begin{pmatrix}
\gamma_0(t) & \gamma_1(t) & \dotsc & \gamma_{2n-1}(t)
\end{pmatrix}. \]
Let us check that our formula for boundary correlations is coherent with the
above theorem, and in fact provide with a constructive alternative proof for
these particular coupling constants. In fact, we prove something stronger:
\begin{lem}
  For all $t\in\R$, there exists a non-zero real number $c(t)$ such that
  \[\forall p\in\Z,\quad\gamma_p(t)=c(t)\gamma_{p+1}^*(t).\]
  In particular, we recover that $\Span(\gamma_R)\cdot S = \Span(\gamma^*_R)$.
\end{lem}
\begin{proof}
Let $p\in \Z$. Suppose $\tau(p)\neq p+1$. Then
$J_{p+1}=(J_p\setminus\{p+1\})\cup\{\tau(p)\}$. If $p$ is even then
\begin{align*}
\frac{\gamma_p(t)}{\gamma^*_{p+1}(t)} & = (k')^{p+n-1}\prod_{j=1}^p\frac{1}{\mathsf{dn}(t-\alpha_j)}\prod_{j=0}^{p-1}\frac{1}{\mathsf{dn}(t-\alpha_j)}\prod_{j\in J_p}\frac{1}{\mathsf{dn}(t-\alpha_j)},
\end{align*}
and
\begin{align*}
\frac{\gamma_{p+1}(t)}{\gamma^*_{p+2}(t)} & = (k')^{p+1+n-1}\prod_{j=1}^{p}\frac{1}{\mathsf{dn}(t-\alpha_j)}\prod_{j=0}^{p+1}\frac{1}{\mathsf{dn}(t-\alpha_j)}\prod_{j\in J_{p+1}}\frac{1}{\mathsf{dn}(t-\alpha_j)} \\
	& = (k')^{p+1+n-1}\prod_{j=1}^{p}\frac{1}{\mathsf{dn}(t-\alpha_j)}\prod_{j=0}^{p+1}\frac{1}{\mathsf{dn}(t-\alpha_j)}\prod_{j\in J_{p}}\frac{1}{\mathsf{dn}(t-\alpha_j)}\frac{\mathsf{dn}(t-\alpha_{p+1})}{\mathsf{dn}(t-\alpha_{\tau(p)})} \\
	& = \frac{\gamma_p(t)}{\gamma^*_{p+1}(t)} \quad \text{ since } \mathsf{dn}(t-\alpha_p)\mathsf{dn}(t-\alpha_{\tau(p)})=k'.
\end{align*}
Similarly when $p$ is odd, we have $\frac{\gamma_p(t)}{\gamma^*_{p+1}(t)}=\frac{\gamma_{p+1}(t)}{\gamma^*_{p+2}(t)}$. We also get the result when $\tau(p)=p+1$ using the fact that $\mathsf{dn}(t-\alpha_p)\mathsf{dn}(t-\alpha_{p+1})=k'$ instead of the previous relation between $J_p$ and $J_{p+1}$. It implies that there exists a constant $c\in \R^{*}$ such that for all $p\in \Z$ we have $\frac{\gamma_p(t)}{\gamma^*_{p+1}(t)}=c$.
\end{proof}

\begin{rem}
When the region is a regular $2n$-gon, the constant $c$ is equal to the product
$\prod_{j=n+1}^{2n-1}\mathsf{dn}(t-\alpha_j)$ and it is immediately clear that
it is independent of $p$.
\end{rem}

\section{Near critical case}

We now explain the computations that one should make in order to study the near
critical case, that is to say the case when $k$ tends to $0$.  Let $p \in [2n]$
and let $t \in \R$. We do a Taylor expansion of $\gamma_p(t,k)$, which is
defined as $\gamma_p(t)$ but we underline the dependence in $k$. First we get,
using Equations~\eqref{eq:taylor_sn_rescaled} and~\eqref{eq:taylor_dn_rescaled},
\begin{equation*}
\mathsf{sn}(t-\alpha_j)  =\sin(t-\alpha_j)(1+\frac{k^2}{4}\cos^2(t-\alpha_j))+O(k^4))
\end{equation*}
and 
\begin{equation*}
\mathsf{dn}(t-\alpha_j)= 1-\frac{1}{2}k^2\sin^2(t-\alpha_j)+O(k^4).
\end{equation*}
Finally, using the fact that $k'=\sqrt{1-k^2}$ we get
\[ \frac{1}{\sqrt{k'}}=\frac{1}{(1-k^2)^{\frac{1}{4}}}=1+\frac{k^2}{4}+O(k^4).\]
Combining all the above expansions and expanding the product, we have the
following expansion for $\gamma_p(t,k)$:
\begin{multline*}
\gamma_p(t,k)  = \gamma_p(t,0) \\+\frac{k^2}{4}\prod_{j\in J_p}\sin(t-\alpha_j)\left[\sum_{j\in J_p}\cos^2(t-\alpha_j)+2\sum_{j=1}^{2\lfloor\frac{p}{2}\rfloor}\cos(2(t-\alpha_j))\right]+O(k^4).
\end{multline*}

Let $R$ be a region $R=(\tau,\bm{\alpha})$ where the angles $(\alpha_i)_{i\in
[2n]}$ are all distinct. Recall that, by Proposition~\ref{prop_K_n}, we get the
spin boundary correlations by computing $(\Gamma(k)K_n)^{-1}\Gamma(k)$, where
$\Gamma(k)$ is the matrix with rows
$\gamma_R(\alpha_{j_1}),\dotsc,\gamma_R(\alpha_{j_n})$ with
$J_1\cup\{1\}:=\{j_1,\dotsc,j_n\}$ and $k$ is the elliptic parameter. In the
critical case, an explicit expression for the coefficients of
$(\Gamma(0)K_n)^{-1}\Gamma(0)$ is known,
see~\cite[Theorem~1.1]{galashin2022formula}. In the general $Z$-invariant case, we
have
\begin{align*}
(\Gamma(k)K_n)^{-1}\Gamma(k) = &\left[\bigl(\Gamma(0)+(\Gamma(k)-\Gamma(0))\bigr)K_n\right]^{-1}\Gamma(k) \\
	 = &\Bigl[(\Gamma(0)K_n)^{-1}-(\Gamma(k)-\Gamma(0))K_n \\
	 & + O(\norme{\Gamma(k)-\Gamma(0)}^2)\Bigr](\Gamma(0)+(\Gamma(k)-\Gamma(0)) \\
	 =& (\Gamma(0)K_n)^{-1}\Gamma(0)+(\Gamma(0)K_n)^{-1}(\Gamma(k)-\Gamma(0)) \\
	 & -(\Gamma(k)-\Gamma(0))K_n\Gamma(0)+O(\norme{\Gamma(k)-\Gamma(0)}^2).
\end{align*}
Notice that we have an explicit expression for every term in this expansion
except $(\Gamma(0)K_n)^{-1}$. But note that for $k=0$ we have
\begin{align*}
\gamma_p(t,0) = & \prod_{j\in J_p}\sin(t-\alpha_j) \\
	= & \prod_{j\in J_p}\bigl(\sin(t)\cos(\alpha_j)-\cos(t)\sin(\alpha_j)\bigr) \\
	= & \sum_{l=1}^{n}(-\cos(t))^{l-1}\sin(t)^{n-l}\sum_{\substack{H\subset J_p \\ |H|=l+1}}\prod_{j\in H}\sin(\alpha_j)\prod_{j\in J_p\setminus H}\cos(\alpha_j).
\end{align*}

Therefore $\Gamma(0)=V\cdot F$ where $V$ is an invertible Vandermonde-like
matrix with entries $((-\cos(t_{j_i}))^{l}\sin(t_{j_i})^{n-l})_{1\leq i,l\leq n}$ and $F$ is the $n\times 2n$ matrix defined by, for all $1\leq l\leq n$ and for all $1\leq p\leq 2n$, 
\[ F_{l,p}:=\sum_{\substack{H\subset J_p \\ |H|=l+1}}\prod_{j\in H}\sin(\alpha_j)\prod_{j\in J_p\setminus H}\cos(\alpha_j). \]

In~\cite{galashin2022formula}, it is proven that $F=D\cdot B$ where $D$ is a
diagonal matrix and $B$ is a $n\times 2n$ matrix with entries
$B_{lj}=z_l^{j-1}$ with $z_l=e^{\frac{\pi i}{2n}(2l-n-1)}$. The inverse of
$BK_n$ can then be computed explicitly, see again~\cite{galashin2022formula} for
details. Therefore $(\Gamma(0)K_n)^{-1}$ also has an explicit expression but
considering the complexity of this expression, we did not manage to get
information from it.

\appendix

\section{Useful identities involving elliptic functions}

We state here facts about Jacobi elliptic functions that we use in this article
(without proofs). More details can be found in~\cite{abramowitz1965handbook} and
\cite{lawden}.

Let $k\in \C$ such that $k^2\in (-\infty, 1)$. Define
$K=K(k):=\int_{0}^{\frac{\pi}{2}}\frac{1}{\sqrt{1-k^2\sin^2(x)}}dx$, it is
called the \emph{complete elliptic integral of the first kind}. Let
$k':=\sqrt{1-k^2}$, and define thedual elliptic parameter $k^*$ by
$k^*:=i\frac{k}{k'}$.

We recall our rescaling convention for elliptic functions: if $p$ and $q$ are
distinct letters among $c,s,d,n$, we set for all $u\in\C$:
\begin{equation*}
  \mathsf{pq}(u)=\mathrm{pq}(\frac{2K(k) u}{\pi}|k),\quad
  \mathsf{pq}^*(u)=\mathrm{pq}(\frac{2K(k^*) u}{\pi}|k^*).
\end{equation*}

\paragraph{Change of arguments}

Jacobi's elliptic functions satisfy addition formula by quarter-periods and
half-periods (see~\cite[Table~16.8]{abramowitz1965handbook} for the unrescaled
version), among which:

\begin{table}[ht!]
\centering
\begin{tabular}{cccc}
\toprule
$u$ & $-u$ & $u\pm \frac{\pi}{2}$ & $u + \pi$ \\ \midrule
$\mathsf{sn}$ & $-\mathsf{sn}$ &$\pm         \mathsf{cd}$ & $-\mathsf{sn}$ \\ \midrule
$\mathsf{cn}$ & $ \mathsf{cn}$ & $\mp k'     \mathsf{sd}$ & $-\mathsf{cn}$ \\ \midrule
$\mathsf{dn}$ & $ \mathsf{dn}$ & $k'\frac{1}{\mathsf{dn}}$ & $\mathsf{dn}$ \\ \midrule
$\mathsf{cd}$ & $ \mathsf{cd}$ &$\mp         \mathsf{sn}$ & $-\mathsf{cd}$ \\ \bottomrule
\end{tabular}
\caption{Formulas for quarter-periods and half-periods}
\label{table_formula}
\end{table}

Some values of these functions are known explicitly:
\begin{equation} \mathsf{sn}(\frac{\pi}{4})=\frac{1}{\sqrt{1+k'}}
  \quad\text{and}\quad \mathsf{sn}(\frac{\pi}{2})=1 ,
\label{value_sn}
\end{equation}
\begin{equation}
\mathsf{cn}(\frac{\pi}{4})=\frac{\sqrt{k'}}{\sqrt{1+k'}},
\label{value_cn}
\end{equation}
\begin{equation}
\mathsf{dn}(\frac{\pi}{4})=\frac{1}{\sqrt{k'}}.
\label{value_dn}
\end{equation}

\paragraph{Addition formulas}
Some versions of addition formulas used in the proof of the main lemma are

\begin{equation}
  \forall u,v \in \C \quad
  \mathsf{dn}(u)\mathsf{sn}(u+v) =
  \mathsf{cn}(u)\mathsf{sn}(v) + \mathsf{sn}(u)\mathsf{cn}(v)\mathsf{dn}(u+v),
\label{sn_formula}
\end{equation}

\begin{equation}
\forall u,v\in\C \quad \mathsf{cn}(u+v)=
\mathsf{cn}(u)\mathsf{cn}(v) - \mathsf{sn}(u)\mathsf{sn}(v)\mathsf{dn}(u+v).
\label{cn_formula}
\end{equation}
They are consequences of more classical addition formulas
(see~\cite[16.17.1--16.17.3]{abramowitz1965handbook}),
and are listed in their unrescaled version as exercises in~\cite[Chapter~2,
Exercise~32]{lawden}.

\paragraph{Dual formulas}
Some formulas that link elliptic functions and dual
elliptic functions are:

\begin{equation}
\forall u \in \C \quad k'\mathsf{sd}(u)=\mathsf{sn}^*(u),
\label{dual_sn}
\end{equation}

\begin{equation}
\forall u \in \C \quad \mathsf{cd}(u)=\mathsf{cn}^*(u),
\label{dual_cn}
\end{equation}

\begin{equation}
\forall u\in\C \quad \frac{1}{\mathsf{dn}(u)}=\mathsf{dn}^*(u).
\label{dual_dn}
\end{equation}
See~\cite[Section~16.10]{abramowitz1965handbook}.

\paragraph{Taylor expansion when $k\rightarrow 0$}

\begin{equation}
K(k) = \frac{\pi}{2}+\frac{\pi}{8}k^2 + O(k^4),
\label{exp_K}
\end{equation}

\begin{equation}
\forall u\in\C \quad \sn(u)=\sin(u)-\frac{1}{4}k^2(u-\sin(u)\cos(u))\cos(u)+O(k^4),
\label{exp_sn}
\end{equation}

\begin{equation}
\forall u \in \C \quad \dn(u)=1-\frac{1}{2}k^2\sin^2(u)+O(k^4).
\label{exp_dn}
\end{equation}
See~\cite[Section~16.13]{abramowitz1965handbook}.
From by~\eqref{exp_K}, we have
\begin{equation*}
  \sin(\frac{2K(k)u}{\pi})=\sin(u)+\frac{k^2}{4}u\cos(u)+O(k^4),
  \cos(\frac{2K(k)u}{\pi})=\cos(u)-\frac{k^2}{4}u\sin(u)+O(k^4).
\end{equation*}
Injecting these expansions in
the rescaled functions $\mathsf{sn}$ and $\mathsf{dn}$,
one gets:

\begin{align}
\mathsf{sn}(u)  = &\sn(\frac{2K(k)}{\pi}u|k) \nonumber\\
                = &
  \sin(\frac{2K(k)u}{\pi})-\frac{k^2}{4}\left(\frac{2K(k)u}{\pi}
	 -
       \sin\bigl(\frac{2K(k)u}{\pi}\bigr)\cos\bigl(\frac{2K(k)u}{\pi}\bigr)\right)\cos\bigl(\frac{2K(k)u}{\pi}\bigr)+O(k^4)
       \nonumber\\
	= & \sin(u)(1+\frac{k^2}{4}\cos^2(u))+O(k^4))
        \label{eq:taylor_sn_rescaled}
\end{align}
and 
\begin{align}
\mathsf{dn}(t-\alpha_j)&= 1-\frac{1}{2}k^2\sin^2(\frac{2K(k)}{\pi}(t-\alpha_j)) + O(k^4)
\nonumber
\\
& = 1-\frac{1}{2}k^2\sin^2(t-\alpha_j)+O(k^4).
        \label{eq:taylor_dn_rescaled}
\end{align}

\bibliographystyle{alpha}
\bibliography{references}

\end{document}